\documentclass[10pt]{article}
\usepackage[a4paper,centering,vscale=0.8,hscale=0.75]{geometry}
\usepackage{mathtools}
\usepackage{amsthm,amssymb}
\usepackage{mathrsfs}
\usepackage{bbm}
\usepackage{color}

\newtheorem{thm}{Theorem}[section]
\newtheorem{lemma}[thm]{Lemma}
\newtheorem{prop}[thm]{Proposition}
\newtheorem{coroll}[thm]{Corollary}
\newtheorem{defi}[thm]{Definition}
\theoremstyle{remark}
\newtheorem{rmk}[thm]{Remark}

\newcommand{\dom}{\mathsf{D}}
\newcommand{\E}{\mathop{{}\mathbb{E}}}
\newcommand{\cF}{\mathscr{F}}
\newcommand{\pG}[1]{{#1}_\Gamma}
\newcommand{\cJ}{\mathscr{J}}
\newcommand{\cL}{\mathscr{L}}
\renewcommand{\P}{\mathbb{P}}
\newcommand{\erre}{\mathbb{R}}
\newcommand{\enne}{\mathbb{N}}

\newcommand{\embed}{\hookrightarrow}
\newcommand{\ind}[1]{\mathbbm{1}_{#1}}
\newcommand{\longto}{\longrightarrow}

\DeclarePairedDelimiter\abs{\lvert}{\rvert}
\DeclarePairedDelimiter\norm{\lVert}{\rVert}
\DeclarePairedDelimiterX\ip[2]{\langle}{\rangle}{#1,#2}

\numberwithin{equation}{section}

\title{Refined existence and regularity results for a class of
  semilinear dissipative SPDEs}
  \author{Carlo Marinelli\thanks{Department of Mathematics, University
      College London, Gower Street, London WC1E 6BT, United
      Kingdom. URL: \texttt{http://goo.gl/4GKJP}} 
  \and Luca Scarpa\thanks{Faculty of Mathematics, University
      of Vienna, Oskar-Morgenstern-Platz 1, 1090 Vienna, Austria. 
      E-mail: \texttt{luca.scarpa@univie.ac.at}, URL: \texttt{https://mat.univie.ac.at/$\sim$scarpa/}}}
\date{December 18, 2019}

\begin{document}
\maketitle

\begin{abstract}
  We prove existence and uniqueness of solutions to a class of
  stochastic semilinear evolution equations with a monotone nonlinear
  drift term and multiplicative noise, considerably extending
  corresponding results obtained in previous work of ours.  In
  particular, we assume the initial datum to be only measurable and we
  allow the diffusion coefficient to be locally
  Lipschitz-continuous. Moreover, we show, in a quantitative fashion,
  how the finiteness of the $p$-th moment of solutions depends on the
  integrability of the initial datum, in the whole range
  $p \in ]0,\infty[$.  Lipschitz continuity of the solution map in
  $p$-th moment is established, under a Lipschitz continuity
  assumption on the diffusion coefficient, in the even larger range
  $p \in [0,\infty[$.  A key role is played by an It\^o formula for
  the square of the norm in the variational setting for processes
  satisfying minimal integrability conditions, which yields pathwise
  continuity of solutions.  Moreover, we show how the regularity of
  the initial datum and of the diffusion coefficient improves the
  regularity of the solution and, if applicable, of the invariant
  measures.
\end{abstract}


\section{Introduction}
\label{sec:intro}
We consider semilinear stochastic partial differential equations on a
smooth bounded domain $D \subseteq \erre^d$ of the form
\begin{equation}
  \label{eq:0}
  dX(t) + AX(t)\,dt + \beta(X(t))\,dt \ni B(t,X(t))\,dW(t),
  \qquad X(0)=X_0,
\end{equation}
where $A$ is a coercive maximal monotone operator on (a subspace of)
$H:=L^2(D)$, $\beta$ is a maximal monotone graph in $\erre \times \erre$
defined everywhere, $W$ is a cylindrical Wiener process on a separable
Hilbert space $U$, and $B$ is a process taking values in the space of
Hilbert-Schmidt operators from $U$ to $L^2(D)$ satisfying a (local)
Lipschitz continuity condition. Precise assumptions on the data of the
problem are given in \S\ref{sec:ass} below.

Assuming that the initial datum $X_0$ has finite second moment and the
diffusion coefficient $B$ is globally Lipschitz continuous, we proved
in \cite{cm:AP18} that equation \eqref{eq:0} admits a unique solution,
in a generalized variational sense, whose trajectories are weakly
continuous in $H$. The contribution of this work is to extend these
results in several directions.
In fact, we show that the solution $X$ is pathwise strongly continuous
in $H$, rather than just weakly continuous. This is possible thanks to
an It\^o-type formula, interesting in its own right, for the square of
the $H$-norm of processes satisfying minimal integrability conditions,
in a variational setting extending the classical one by Pardoux
\cite{Pard}.
The pathwise strong continuity of solutions to stochastic equations
with singular drift is often a difficult problem, and it was left as
an open problem both in the case of semilinear equations (see
\cite{Barbu:Lincei}) as well as of porous media equations (see
\cite{BDPR-porous}). This trajectorial regularity result is essential
to prove that the Markovian semigroup generated by the solution to
\eqref{eq:0} is strongly continuous, for which we refer to
\cite{cm:inv}.
Moreover, the pathwise strong continuity allows us to prove that
existence and uniqueness of solutions to \eqref{eq:0} continues to
hold under much weaker assumptions on the initial datum and on the
diffusion coefficient. In particular, $X_0$ needs only be measurable
and $B$ can be locally Lipschitz-continuous with linear
growth. Denoting by $\Omega$ the underlying probability space, the
solution map $X_0 \mapsto X$ is thus defined on $L^0(\Omega;H)$, with
codomain contained in $L^0(\Omega;E)$, where $E$ is a suitable path
space. By the results of \cite{cm:AP18} we also have that the solution
map restricted to $L^2(\Omega;H)$ has codomain contained in
$L^2(\Omega;E)$. As a further result, we extrapolate these mapping
properties to the whole range of exponents $p \in [0,\infty[$, that
is, we show that if $X_0 \in L^p(\Omega;H)$ then $X \in L^p(\Omega;E)$
for every positive finite $p$, and we provide an explicit upper bound
on the $L^p(\Omega;E)$-norm of the solution in terms of the
$L^p(\Omega;H)$-norm of the initial datum. If, in addition, $B$ is
Lipschitz-continuous, we show that the solution map is
Lipschitz-continuous from $L^p(\Omega;H)$ to $L^p(\Omega;E)$ for all
$p \in [0,\infty[$. In the particular case $p=0$, this implies that
solutions converge uniformly on $[0,T]$ in probability if the
corresponding initial data converge in probability.

The existence of higher moments of the solution in
suitable path spaces is essential for various applications, for
instance in optimal control problems of stochastic evolution systems
of reaction-diffusion type with singular terms. Problems of this
type arise naturally in phase-transition modelling and tumor-growth
dynamics (see, e.g.,\cite{garcke-lam,garcke-lam-roc,scar-OCSCH}).
Indeed, a classical way to prove first-order conditions for optimality
of a certain given cost functional is to study the differentiability
of the so-called control-to-state map.  This can be obtained provided
that a refined continuous dependence result on the controls holds. To
this end, due to the high nonlinearity of the equations involved, one
is usually forced to prove boundedness of the solutions in function
spaces with higher integrability in $\Omega$ (see, e.g.,
\cite{scar-OCSCH}).  For these reasons, the possibility of relating
the existence of moments (in probability) of the solutions to the
integrability of the data is fundamental. However, while higher-moment
estimates are well known for equations with well-behaved
nonlinearities, in case of singular potentials as in \eqref{eq:0} no
result is currently available in literature.  The possibility of
including multivalued graphs $\beta$ as in \eqref{eq:0} is crucial in
diffuse-interface and phase-change modelling, and the results proved
in this paper hence constitute a first reference in this direction.
Let us also mention that the availability of an $L^p$-theory is also
crucial in the context of numerical approximation of SPDEs, as for
example the convergence rate and the choice of the discretization
scheme usually depend on the integrability properties of the solutions
(see, e.g., in the context of numerical discretization of SPDEs,
\cite[Introduction]{krylov-AA}, \cite{gyon:lattice}, and
\cite{cohen-quer}.
  
Finally, we show how the smoothness of the solution improves (as well
as of invariant measures, if they exist) if the initial datum and the
diffusion coefficient are smoother, without any further regularity
assumption on the (possibly singular) monotone drift term $\beta$. For
example, if $A$ (better said, the part of $A$ in $H$) is self-adjoint,
the solution has paths belonging to the domain of $A$ in $H$ if $X_0$
and $B$, roughly speaking, take values in the domain of
$A^{1/2}$. This implies that $X$ is a strong solution in the classical
sense, not just in the variational one. In particular, if \eqref{eq:0}
represents the abstract formulation of an initial-boundary value
problem on $D$, our result provides sufficient conditions on the data
implying higher regularity in space of the solution. The main novelty
is that this extra regularity can be obtained \emph{irrespectively} of
the bahaviour of the nonlinear term $\beta$, meaning that one can
obtain higher regularity in space (thus also better compactness
properties) without necessarily ``smoothing'' out the singularities of
$\beta$. Such a result is new in the literature, to the best of our
knowledge, and is essential in the context of optimal control problems
of stochastic systems in phase-field modelling (see
\cite{scar-OCSCH}). In fact, on the one hand, equations with singular
$\beta$ are necessary to model phase-change phenomena, and, on the
other hand, the possibility of proving necessary conditions for
optimality hinges on the availability of spatial regularity
results. For this reason, our contribution is a first step in this
direction, apart of providing new results in the regularity theory of
singular dissipative stochastic PDEs.

\medskip

Let us comment now more in detail on the current available literature.

In the classical variational theory of SPDEs, existence and uniqueness
of solutions under a local Lipschitz condition on $B$ and
measurability of $X_0$ were obtained by Pardoux in \cite{Pard}. Our
results do not follow from his, however, as equation \eqref{eq:0}
cannot be cast in the usual variational setting. Stochastic equations
where \emph{all} nonlinear terms are locally Lipschitz-continuous have
been considered in the semigroup approach (see, e.g., \cite{KvN2} and
references therein), but our existence results are not covered, as
$\beta$ can be discontinuous and have arbitrary growth.  Moreover, the
properties of the solution map between $L^p(\Omega;H)$ and
$L^p(\Omega;E)$ do not seem to have been addressed even in the
classical variational setting. On the other hand, the continuity of
the solution map in the case $p=0$ for ordinary SDEs in $\erre^n$ with
Lipschitz coefficients has been studied, also with very general
semimartingale noise (see, e.g., \cite{Eme:stab}).

The techniques used to obtain the above results are mostly classical:
we truncate the initial datum and the diffusion coefficient in a
suitable way, so that existence of solutions can be proved locally
(i.e. on a stochastic interval). Uniform estimates on local solutions
allow then to extend them to the whole interval $[0,T]$. It
is however worth stressing the fact that these arguments are
relatively simple only when applied to classes of equations for which
existence of strong solutions with continuous trajectories is known
(e.g. to equations with a polynomially growing drift). On the other
hand, in the case of singular nonlinearities as in \eqref{eq:0}
several non-trivial difficulties appear: for instance, if only weak
pathwise continuity is known (as is the case for \eqref{eq:0}, as well
as for other equations -- cf. the already mentioned \cite{BDPR-porous,
  cm:AP18}), it is not clear at all how to \emph{define} local
solutions.

The theory of stochastic evolution equations with monotone drift has
considerably developed since the initial contributions by Pardoux,
Krylov and Rozovski\u{\i} \cite{KR-spde,Pard}. In fact, while in the
deterministic case general well-posedness theories for evolution
equations with $m$-accretive or maximal monotone operators are
available (see, e.g.,~\cite{Barbu:type}), the picture is much less
complete in the stochastic setting, where a general well-posedness
theory is currently out of reach. Nonetheless, several results are
available for specific classes of equations.
For instance, existence of strong solution has been proved in
\cite{Gess} for stochastic equations on a Hilbert space whose drift is
the subdifferential of a potential satisfying a sub-homogeneity
assumption (an exponentially growing drift is then not
allowed). Probabilistically weak solutions, defined in terms of
stochastic variational inequalities, have been shown to exist in
\cite{BenRas}, with no restrictive assumptions on the potential. Let
us also mention, parenthetically, that concepts of solution based on
stochastic variational inequalities have been proposed also for
equations with drift in divergence form and of fast-diffusion type
(see \cite{GessRoc-fast,GessRoc}).

All of the above results are not applicable to semilinear equations
such as \eqref{eq:0}, if neither growth nor coercivity assumptions on
$\beta$ are assumed. Well-posedness in the strong sense and ergodicity
properties for this class of equations have been obtained in
\cite{cm:inv, cm:jump, cm:AP18}, where no growth conditions on the
drift are needed.  Several analogous results have also been obtained
for fully nonlinear equations in divergence form with singular drift,
again without any growth assumptions, in \cite{cm:note, cm:div, luca}.

\medskip

The rest of the text is organized as follows. In \S\ref{sec:ass} we
state the main assumptions and we recall the well-posedness result for
\eqref{eq:0} obtained in \cite{cm:AP18}.
In \S\ref{sec:cont} we prove a generalized It\^o formula for the
square of the norm, as well as the strong pathwise continuity of
solutions.
In \S\ref{sec:X0} we prove existence and uniqueness of strong
variational solutions to \eqref{eq:0} assuming first that $B$ is
locally Lipschitz-continuous with linear growth and that $X_0$ is
square integrable, hence removing the latter assumption in a second
step, allowing $X_0$ to be merely measurable. While in the former case
solutions have finite second moment, in the latter case one needs to
work with processes that are just measurable (in $\omega$), so that
uniqueness has to be proved in a much larger space. This is achieved
by a suitable application of the It\^o formula of {\S}\ref{sec:cont}
and stopping arguments.
In \S\ref{sec:mom} we show that $X_0$ having finite $p$-th moment
implies that the solution belongs to a space of processes with finite
$p$-moment as well, with explicit control of its norm. The Lipschitz
continuity of the solution map is then established in a particular case.
Further regularity of the solution and of invariant measures is
obtained in the last section, under additional regularity assumptions
on $X_0$ and $B$.

\subsection*{Acknowledgments}
The authors thank David Cohen for some comments and bibliographycal
suggestions.
The first-named author gratefully acknowledges the hospitality of the
Interdisziplin\"ares Zentrum f\"ur Komplexe Systeme (IZKS) at the
University of Bonn, Germany.
The second-named author gratefully acknowledges the Austrian Science Fund (FWF)
project M 2876.


\section{Assumptions and preliminaries}
\label{sec:ass}
\subsection{Notation and terminology}
Given a Banach space $E$, its (topological) dual will be denoted by
$E'$. Given a further Banach space $F$, the (Banach) space of linear
bounded operators from $E$ to $F$ will be denoted by $\cL(E,F)$. If
$E$ and $F$ are Hilbert spaces, $\cL^2(E,F)$ stands for the space of
Hilbert-Schmidt operators from $E$ to $F$. We recall that
Hilbert-Schmidt operators form a two-sided ideal on linear bounded
operators.

A graph $\gamma$ in $E$ is a subset of $E\times E$ and the domain of
$\gamma$ is defined as
$\dom(\gamma) := \{ x\in E: \exists\,y\in E: (x,y) \in \gamma \}$.  We
shall identify linear unbounded operators between Banach spaces with
their graphs, as usual.  If $E$ is a Hilbert space, $\gamma$ is
monotone if $(x_1,y_1)$, $(x_2,y_2) \in \gamma$ implies
$\ip{y_2-y_1}{x_2-x_1}_E \geq 0$, where $\ip{\cdot}{\cdot}_E$ is the
scalar product in $E$. The notion of maximal monotone graph is
immediate once graphs are ordered by inclusion.

We shall use the standard notation of stochastic calculus (see, e.g.,
\cite{Met}). In particular, given a c\`adl\`ag process $Y$ with values
in a separable Banach space $E$, the process $Y^*$ is defined as
$Y^*(t):=\sup_{s \in [0,t]} \norm{Y(s)}$. For notational convenience,
we shall also denote the time index as a subscript rather than within
parentheses. Moreover, a process $Y$ stopped at a stopping time $S$ is
denoted by $Y^S$, and the stochastic integral of $K$ with respect to a
local martingale $M$ is denoted by $K \cdot M$.

\subsection{Assumptions}
\label{ssec:ass}
Let $D$ be a bounded domain in $\erre^d$ with smooth boundary, and $V$
a real separable Hilbert space densely, continuously, and compactly
embedded in $H:=L^2(D)$. The scalar product and the norm of $H$ will
be denoted by $\ip{\cdot}{\cdot}$ and $\norm{\cdot}$,
respectively. Identifying $H$ with its dual $H'$, the triple
$(V,H,V')$ is a so-called Gelfand triple: the duality form between $V$
and $V'$ extends the scalar product of $H$,
i.e. $\ip{v}{w}={}_V\ip{v}{w}_{V'}$ for any $v$,
$w \in H$. For this reason, we shall simply denote the duality form of
$V$ and $V'$ by the same symbol used for the scalar product in $H$.

The following assumptions on the linear operator $A \in \cL(V,V')$
will be tacitly assumed to hold throughout the whole text:
\begin{itemize}
\item[(i)] there exists $C>0$ such that $\ip{Av}{v} \geq C\norm{v}_V^2$
  for every $v\in V$;
\item[(ii)] the part of $A$ in $H$ can be extended to an $m$-accretive
  operator $A_1$ on $L^1(D)$;
\item[(iii)] for every $\delta>0$, the resolvent $(I+\delta A_1)^{-1}$
  is sub-Markovian, i.e. for every $f\in L^1(D)$ such that
  $0 \leq f \leq 1$ a.e. on $D$, one has
  $0 \leq (I+\delta A)^{-1}f \leq 1$ a.e. on $D$;
\item[(iv)] there exists $m \in \enne$ such that
  $(I+\delta A_1)^{-m} \in \cL(L^1(D), L^\infty(D))$.
\end{itemize}
We shall occasionally refer to hypothesis (i) as coercivity of $A$,
and to hypothesis (iv) as ultracontractivity of the resolvent of
$A_1$.

\smallskip

Let us now state the assumptions on the nonlinear part of the drift:
$\beta \subset \erre \times \erre$ is a maximal monotone graph such
that $0 \in \beta(0)$ and $\dom(\beta)=\erre$.  Let
$j:\erre \to [0,+\infty)$ be the unique convex lower-semicontinuous
function such that $j(0)=0$ and $\beta=\partial j$, where $\partial$
stands for the subdifferential in the sense of convex analysis. We
assume that
\[
  \limsup_{|r|\to\infty} \frac{j(r)}{j(-r)} < \infty.
\]
Denoting the Moreau-Fenchel conjugate of $j$ by $j^*$, the fact that
$\dom(\beta)=\erre$ is equivalent to the superlinearity of $j^*$ at
infinity, i.e. to
\[
  \lim_{|r|\to\infty}\frac{j^*(r)}{|r|} = +\infty.
\]
For a comprehensive treatment of maximal monotone operators and their
connection with convex analysis we refer to, e.g.,
\cite{Barbu:type}. Here we limit ourselves to recalling that, for any
maximal monotone graph $\gamma$ on a Hilbert space $E$, its resolvent
and Yosida approximation of $\gamma$ are defined as
$(I+\lambda\gamma)^{-1}$ and
\[
  \gamma_\lambda := \frac{1}{\lambda}
  \bigl( I - (I+\lambda\gamma)^{-1} \bigr),
\]
respectively, that both are continuous operators on $E$, and that the
former is a contraction, while the latter is Lipschitz-continuous with
Lipschitz constant bounded by $1/\lambda$.

\smallskip

Let $(\Omega,\cF,\P)$ be a probability space, endowed with a
right-continuous and completed filtration $(\cF_t)_{t \in [0,T]}$, on
which a cylindrical Wiener process $W$ on a real separable Hilbert
space $U$ is defined. The diffusion coefficient
\[
  B:\Omega \times[0,T] \times H \to \cL^2(U,H)
\]
is assumed to be such that $B(\cdot,\cdot,x)$ is progressively
measurable for every $x\in H$, and to grow at most linearly in its
third argument, uniformly with respect to the others. That is, we
assume that there exists a constant $N$ such that
\[
  \norm[\big]{B(t,\omega,x)}_{\cL^2(U,H)} \leq
  N \bigl( 1+\norm{x} \bigr)
\]
for all $(\omega,t,x) \in \Omega \times [0,T] \times H$. In addition
to this, we shall consider two different assumptions, namely
\begin{itemize}
\item[(B1)] $B$ is Lipschitz continuous in its third argument,
  uniformly with respect to the others, i.e.
  \[
    \norm[\big]{B(\omega,t,x) - B(\omega,t,y)}_{\cL^2(U,H)}
    \leq N \norm{x-y}
  \]
  for all $(\omega,t) \in \Omega \times [0,T]$ and $x$, $y \in H$.
\item[(B2)] $B$ is locally Lipschitz continuous in its third argument,
  uniformly with respect to the others, i.e. there exists a function
  $R \mapsto N_R: \erre_+ \to \erre_+$ such that
  \[
    \norm[\big]{B(\omega,t,x) - B(\omega,t,y)}_{\cL^2(U,H)}
    \leq N_R \norm{x-y}
  \]
  for all $(\omega,t) \in \Omega \times [0,T]$ and $x$, $y \in H$ with
  $\norm{x}$, $\norm{y} \leq R$.
\end{itemize}

\smallskip

Finally, $X_0$ is assumed to be an $H$-valued $\cF_0$-measurable
random variable.

\smallskip

Let us now define the concept of solution to equation \eqref{eq:0}.
\begin{defi}
  \label{def:sol}
  A strong solution to \eqref{eq:0} is a pair $(X,\xi)$, where $X$ is
  a $V$-valued adapted process and $\xi$ is an $L^1(D)$-valued
  predictable process, such that, $\P$-almost surely,
    \begin{gather*}
    X \in L^\infty(0,T; H) \cap L^2(0,T; V), \qquad
    \xi \in L^1(0,T; L^1(D)),\\
    \xi \in \beta(X) \quad\text{a.e.~in } (0,T)\times D,
  \end{gather*}
  and
  \[
    X(t) + \int_0^t AX(s)\,ds + \int_0^t\xi(s)\,ds
    = X_0 + \int_0^tB(s,X(s))\,dW(s)
  \]
  in $V'\cap L^1(D)$ for all $t \in [0,T]$.
\end{defi}
It is convenient to introduce the family of sets $(\cJ_p)_{p\geq 0}$
as follows:
\[
  \cJ_p \subset \Bigl( L^p(\Omega; C([0,T];H))
  \cap L^p(\Omega; L^2(0,T;V)) \Bigr)
  \times L^{p/2}(\Omega; L^1( (0,T) \times D)
\]
formed by processes $(\phi,\psi)$ such that $\phi$ is adapted with
values in $V$, $\psi$ is predictable with values in $L^1(D)$,
$\psi \in \beta(\phi)$ a.e. in $\Omega \times (0,T) \times D$, and
$j(\phi)+j^*(\psi) \in L^{p/2}(\Omega; L^1((0,T) \times D)$.

The following well-posedness result has been proved in
\cite{cm:AP18}. Just for the purposes of this statement, we shall
denote the space $\cJ_2$ with $L^\infty(0,T;H)$ in place of
$C([0,T];H)$ by $\tilde{\cJ}_2$.
\begin{thm}
  \label{thm:WP}
  If $X_0 \in L^2(\Omega,\cF_0;H)$ and $B$ satisfies the global Lipschitz
  condition \emph{(B1)}, then there exists a unique strong solution
  $(X,\xi)$ to \eqref{eq:0} belonging to $\tilde{\cJ}_2$.  Furthermore, the
  trajectories of $X$ are weakly continuous in $H$ and the solution
  map
  \begin{align*}
    L^2(\Omega;H) &\longto L^2(\Omega; L^\infty(0,T; H)) \cap
    L^2(\Omega; L^2(0,T; V))\\
    X_0 &\longmapsto X
  \end{align*}
  is Lipschitz-continuous.
\end{thm}

Our main result is the following far-reaching extension of
Theorem~\ref{thm:WP}: under the more general local Lipschitz
continuity assumption (B2), for any $X_0\in L^p(\Omega,\cF_0,\P;H)$,
$p \in [0,\infty[$, there exists a strong solution $(X,\xi)$ belonging
to $\cJ_p$, which is unique in $\cJ_0$. In particular, the
trajectories of $X$ are strongly continuous in $H$. Precise statements
and proofs are given in {\S}\ref{sec:X0}.


\section{Pathwise continuity via a generalized It\^o formula}
\label{sec:cont}
In this section we prove that, under the assumptions of
Theorem~\ref{thm:WP}, the unique strong solution $(X,\xi)$ in $\cJ_2$
to \eqref{eq:0} is such that $X$ admits a modification with strongly
continuous trajectories in $H$, rather than just weakly continuous. To
this purpose, we need a generalized It\^o's formula for the square of
the norm under minimal integrability assumptions, that will play a
fundamental role throughout.

We first need some preparations.
Let us recall that the part of $A$ in $H$ is the linear (unbounded)
operator on $H$ defined by $A_2 := A \cap (V \times H)$. In particular,
\[
  \dom(A_2) = \bigl\{ u \in V:\, Au \in H \bigr\}
  \qquad \text{and} \qquad
  A_2u = Au \quad \forall  u \in \dom(A_2).
\]
It is well known (see, e.g., \cite{ISEM18}) that $A_2$ is closed and that
$\dom(A_2)$ is a Banach space with respect to the graph norm
\[
  \norm{u}_{\dom(A_2)}^2:= \norm{u}^2 + \norm{Au}^2.
\]
Moreover, $\dom(A_2)$ is continuously and densely embedded in $V$.

\begin{lemma}
  \label{lm:cvV}
  Let $v \in V$ and $v_\lambda := (I+\lambda A_1)^{-1}v$. Then
  $v_\lambda \to v$ in $V$ as $\lambda \to 0$.
\end{lemma}
\begin{proof}
  Let $v\in V$ and $\varepsilon>0$: since $\dom(A_2)$ is densely
  embedded in $V$, we can choose $u \in \dom(A_2)$ such that
  $\norm{v-u}_V < \varepsilon$.  Setting
  $u_\lambda:=(I+\lambda A_1)^{-1}u$, we have
  \[
    \norm{v-v_\lambda}_V \leq \norm{v-u}_V + \norm{u-u_\lambda}_V
    + \norm{u_\lambda - v_\lambda}_V.
  \]
  Since $u,v\in V$, we have $u_\lambda-v_\lambda=(I+\lambda A_2)^{-1}(u-v)$, and 
  recalling that $A_2$ is the part of $A$ in $H$ we have
  \[
  (u_\lambda-v_\lambda) + \lambda A(u_\lambda-v_\lambda)=u-v,
  \]
  where the identity holds in $V$ as well.  Taking the duality product
  with $A(u_\lambda-v_\lambda) \in V'$, by coercivity and boundedness
  of $A$ it follows that
  \[
    \prescript{}{V'}{\ip[\big]{A(u_\lambda-v_\lambda)}{u_\lambda-v_\lambda}_V}
    + \lambda \prescript{}{V'}{%
      \ip[\big]{A(u_\lambda-v_\lambda)}{A(u_\lambda-v_\lambda)}_V} \geq
    C\norm[\big]{u_\lambda-v_\lambda}_V^2
    + \lambda\norm[\big]{A(u_\lambda-v_\lambda)}^2
  \]
  and
  \[
    \prescript{}{V'}{\ip{A(u_\lambda-v_\lambda)}{u}}_V \leq
    \norm{A}_{\cL(V,V')}\norm{u_\lambda-v_\lambda}_V\norm{u}_V,
  \]
  hence
  \[
    C\norm{u_\lambda-v_\lambda}_V^2 + \lambda\norm{A(u_\lambda-v_\lambda)}^2 \leq
    \norm{A}_{\cL(V,V')}\norm{u_\lambda-v_\lambda}_V\norm{u}_V,
  \]
  which implies that there exists a constant $N>0$, independent of
  $\lambda$, such that
  \[
    \norm{u_\lambda-v_\lambda}_V\leq N\norm{u-v}_V,
  \]
  or, equivalently, that $(I+\lambda A_1)^{-1}$ is uniformly bounded
  in $V$ with respect to $\lambda$. This implies that
  \[  
  \norm{u_\lambda - v_\lambda}_V\leq N\norm{u-v}_V\leq N\varepsilon.
  \]
  It remains to estimate the term $\norm{u-u_\lambda}_V$.  Since
  $u \in \dom(A_2)$ and
  \[
  u_\lambda := (I+\lambda A_1)^{-1}u = (I+\lambda A_2)^{-1}u,
  \]
  one has $u_\lambda \in \dom(A_2^2)$, hence, recalling that $A_2$ is
  the part of $A$ in $H$,
  \[
  Au_\lambda + \lambda A(Au_\lambda) = Au
  \]
  in $H \embed V'$. Taking the duality pairing with $Au_\lambda \in
  \dom(A_2) \embed V$, one has
  \[
  \prescript{}{V'}{\ip[\big]{Au_\lambda}{Au_\lambda}_V}
  + \lambda \prescript{}{V'}{\ip[\big]{A(Au_\lambda)}{Au_\lambda}_V}
  = \prescript{}{V'}{\ip[\big]{Au}{Au_\lambda}_V},
  \]
  where
  \begin{gather*}
    \prescript{}{V'}{\ip[\big]{Au_\lambda}{Au_\lambda}_V} =
    \norm[\big]{Au_\lambda}^2, \qquad
    \prescript{}{V'}{\ip[\big]{A(Au_\lambda)}{Au_\lambda}_V}
    \geq C \norm[\big]{Au_\lambda}_V^2,\\
    \prescript{}{V'}{\ip[\big]{Au}{Au_\lambda}_V} =
    \ip[\big]{Au}{Au_\lambda} \leq \frac12 \norm[\big]{Au}^2 + \frac12
    \norm[\big]{Au_\lambda}^2,
  \end{gather*}
  hence
  \[
    \norm[\big]{Au_\lambda}^2 + \lambda C \norm[\big]{Au_\lambda}_V^2
    \leq \frac12 \norm[\big]{Au}^2 + \frac12
    \norm[\big]{Au_\lambda}^2,
  \]
  which implies that
  $\sqrt{\lambda} \norm[\big]{Au_\lambda}_V \leq N \norm[\big]{Au}$,
  with a constant $N$ independent of $\lambda$. Therefore, since
  $u \in \dom(A_2)$,
  \[
    \norm[\big]{u_\lambda - u}_V = \lambda \norm[\big]{Au_\lambda}_V
    \leq N \sqrt{\lambda} \norm[\big]{Au}.
  \]
  Choosing $\lambda$ such that
  $N \sqrt{\lambda} \norm[\big]{Au} < \varepsilon$, one has then
  \[
    \norm{v_\lambda-v}_V < (2+N)\varepsilon,
  \]
  from which the conclusion follows by arbitrariness of $\epsilon$.
\end{proof}

We recall that (see, e.g., \cite{KPS}) if two Banach spaces $F$ and
$G$ are continuously embedded in a separated topological vector space
$E$, their sum $F+G$ is defined as the subspace of $E$
\[
F+G := \bigl\{ u \in E:\, \exists f \in F, \, g \in G: \, u = f+g \bigr\}.
\]
Endowed with the norm
\[
\norm[\big]{u}_{F+G} := \inf_{u=f+g}\bigl( \norm{f}_F + \norm{g}_G \bigr),
\]
$F+G$ is a Banach space. Similarly, the intersection $F \cap G$ is
also a Banach space if endowed with the norm
\[
\norm[\big]{u}_{F \cap G} := \norm{u}_F + \norm{u}_G.
\]
Moreover, if $F \cap G$ is dense in both $F$ and $G$, then $F'$ and
$G'$ are continuously embedded in $(F \cap G)'$, and
$(F+G)'=F' \cap G'$.
In the following we shall deal with $F:=L^1(0,T; H)$ and
$G:=L^2(0,T; V')$, so that as ambient space $E$ one can simply take
$L^1(0,T; V')$. In this case $F \cap G$ is dense in both $F$ and $G$,
hence, by reflexivity of $V$,
\[
\bigl( L^1(0,T;H) + L^2(0,T;V') \bigr)' = L^\infty(0,T;H) \cap L^2(0,T;V).
\]

\begin{thm}
  \label{thm:Ito}
  Let $Y$, $v$ and $g$ be adapted processes such that 
  \begin{gather*}
    Y \in L^0(\Omega; L^\infty(0,T; H)\cap L^2(0,T; V)), \\\
    v \in L^0(\Omega; L^1(0,T; H) + L^2(0,T; V')), \\
    g \in L^0(\Omega; L^1(0,T; L^1(D))), \\
    \exists\,\alpha>0:\quad j(\alpha Y) + j^*(\alpha g) \in
    L^0(\Omega;L^1((0,T) \times D)).
  \end{gather*}
  Moreover, let $Y_0 \in L^0(\Omega,\cF_0; H)$ and $G$ be a
  progressive $\cL^2(U,H)$-valued process such that
  \[
  G \in L^0(\Omega; L^2(0,T;\cL^2(U,H))).
  \]
  If
  \[
    Y(t) + \int_0^tv(s)\,ds + \int_0^t g(s)\,ds = Y_0 +
    \int_0^tG(s)\,dW(s) \qquad \forall t \in [0,T] \quad \P\text{-a.s.}
  \]
  in $V' \cap L^1(D)$, then
  \begin{align*}
    &\frac12 \norm{Y(t)}^2 + \int_0^t\ip[\big]{v(s)}{Y(s)}\,ds
      + \int_0^t\!\!\int_D g(s,x) Y(s,x)\,dx\,ds \\
    &\hspace{3em} = \frac12 \norm{Y_0}^2
      + \frac12\int_0^t \norm{G(s)}^2_{\cL^2(U,H)}\,ds
      + \int_0^t Y(s)G(s)\,dW(s)
  \end{align*}
  for all $t \in [0,T]$ with probability one.
\end{thm}

\begin{proof}
  Since the resolvent of $A_1$ is
  ultracontractive by assumption, there exists $m\in\enne$ such that
  \[  
  (I+\delta A_1)^{-m}: L^1(D) \to H \qquad\forall \delta>0.
  \]
  Using a superscript $\delta$ to denote the action of
  $(I+\delta A_1)^{-m}$, we have
  \[
    Y^\delta(t) + \int_0^tv^\delta(s)\,ds + \int_0^t g^\delta(s)\,ds
    = Y^\delta_0 + \int_0^tG^\delta(s)\,dW(s)
  \]
  where $g^\delta \in L^1(0,T; H)$, hence the classical It\^o's
  formula yields, for every $\delta>0$,
  \begin{align*}
    &\frac12 \norm{Y^\delta(t)}^2
      + \int_0^t\ip[\big]{v^\delta(s)}{Y^\delta(s)}\,ds
      + \int_0^t\!\!\int_D g^\delta(s,x) Y^\delta(s,x)\,dx\,ds \\
    &\hspace{3em} = \frac12 \norm{Y^\delta_0}^2
      + \frac12\int_0^t \norm{G^\delta(s)}^2_{\cL^2(U,H)}\,ds
      + \int_0^t Y^\delta(s)G^\delta(s)\,dW(s).
  \end{align*}
  Let us pass to the limit as $\delta \to 0$. Since the resolvent of
  $A_1$ coincides on $H$ with the resolvent of $A_2$, which converges
  to the identity in $\cL(H)$ in the strong operator topology, we
  immediately infer that
  \begin{alignat*}{3}
    Y^\delta(t) &\longto Y(t)  & &\quad\text{in } H \quad\forall t\in[0,T],\\
    g^\delta &\longto g  & &\quad\text{in } L^1(0,T; L^1(D)),\\
    Y_0^\delta &\longto Y_0  & &\quad\text{in } H,\\
    G^\delta &\longto G  & &\quad\text{in } L^2(0,T; \cL^2(U,H))
  \end{alignat*}
  where the last statement, which follows by well-known continuity properties
  of Hilbert-Schmidt operators, also implies
  \[
    \int_0^t \norm{G^\delta(s)}^2_{\cL^2(U,H)}\,ds \longto
    \int_0^t \norm{G(s)}^2_{\cL^2(U,H)}\,ds.
  \]
  Moreover, by the previous lemma we have
  \[
    Y^\delta \longto Y  \quad\text{in } L^2(0,T; V),
  \]
  and $Y \in L^\infty(0,T;H)$ and the contractivity in $H$ of the
  resolvent of $A_1$ immediately imply, by the dominated convergence
  theorem, that $Y^\delta \to Y$  weakly* in $L^\infty(0,T;H)$.
  Therefore, by reflexivity of $V$,
  \[
    Y^\delta \longto Y  \quad\text{weakly* in } 
    L^\infty(0,T;H) \cap L^2(0,T;V).
  \]
  Since $v \in L^1(0,T; H) + L^2(0,T; V')$, we have that $v=v_1+v_2$,
  with $v_1 \in L^1(0,T; H)$ and $v_2 \in L^2(0,T; V')$.  In this
  case $v^\delta$ has to be interpreted as
  \[
  v^\delta:=(I+\delta A_1)^{-m}v_1 + (I+\delta A)^{-m}v_2.
  \]
  Note that this is very natural since $A_1$ and $A$ coincide on
  $\dom(A_1) \cap V$.  By the properties of the resolvent it easily
  follows that
  \[
  v_1^\delta \longto v_1  \quad \text{ in } L^1(0,T; H).
  \]
  Moreover, since $A^{-1}v_2 \in L^2(0,T; V)$ and
  $A^{-1}v_2^\delta = (I+\delta A_2)^{-m}A^{-1}v_2$, by
  Lemma~\ref{lm:cvV} we have that $A^{-1}v_2^\delta \to A^{-1}v_2$ in
  $L^2(0,T; V)$, hence also, by continuity of $A$,
  \[
  v_2^\delta \longto v_2  \quad \text{ in } L^2(0,T; V').
  \]
  The convergences of $v^\delta$ and $Y^\delta$ just proved thus imply
  \[
    \int_0^t\ip[\big]{v^\delta(s)}{Y^\delta(s)}\,ds \longto
    \int_0^t\ip[\big]{v(s)}{Y(s)}\,ds
  \]
  for all $t \in [0,T]$.

  We are now going to prove that
  $\bigl( (Y^\delta G^\delta) \cdot W - (YG) \cdot W \bigr)^*_T \to 0$
  in probability. Setting $M_\delta := (Y^\delta G^\delta) \cdot W$
  and $M := (YG) \cdot W$, it is well known that it suffices to show
  that the quadratic variation of $M_\delta-M$ converges to $0$ in
  probability. One has
  \begin{align*}
    \bigl[ M_\delta-M,M_\delta-M \bigr]
    &= \norm[\big]{Y^\delta G^\delta - YG}^2_{L^2(0,T;\cL^2(U,\erre))}\\
    &\leq \norm[\big]{Y^\delta G^\delta - Y^\delta G}^2_{L^2(0,T;\cL^2(U,\erre))}
      + \norm[\big]{Y^\delta G - YG}^2_{L^2(0,T; \cL^2(U,\erre))}\\
    &\leq\norm{Y}^2_{L^\infty(0,T; H)}\norm{G^\delta-G}^2_{L^2(0,T; \cL^2(U,H))}+
      \norm{Y^\delta G - YG}^2_{L^2(0,T; \cL^2(U,\erre))},
  \end{align*}
  where the convergence to zero of the first term in the last
  expression has already been proved, and
  \[
    \norm{Y^\delta G - YG}^2_{L^2(0,T; \cL^2(U,\erre))}%
    \leq \int_0^T \norm[\big]{Y^\delta(s)-Y(s)}^2%
    \norm[\big]{G(s)}^2_{\cL^2(U,H)}\,ds \longto 0,
  \]
  by the dominated convergence theorem, because $Y^\delta \to Y$
  pointwise in $H$ and
  $\norm{Y^\delta - Y} \leq 2 \norm{Y} \in L^\infty(0,T)$. We have
  thus shown that
  \[
    \int_0^\cdot Y^\delta(s) G^\delta(s)\,dW(s) \longto 
    \int_0^\cdot Y(s) G(s)\,dW(s)
  \]
  in probability, hence $\P$-a.s. along a subsequence of $\delta$.

  Finally, it is clear that $Y^\delta g^\delta \to Yg$ in measure in
  $(0,T) \times D$, and that, thanks to the assumptions on $j$,
  \[
    \pm \alpha^2 Y^\delta g^\delta \leq %
    j(\pm\alpha Y^\delta) + j^*(\alpha g^\delta) %
    \lesssim 1 + j(\alpha Y^\delta) + j^*(\alpha g^\delta),
  \]
  where the second inequality follows from the fact that, thanks to
  the assumption on the growth of $j$ at $\infty$, there exists a
  constant $M>0$ such that
  \[
  j(r) \leq M \bigl( 1+ j(-r) \bigr) \qquad \forall r \in \erre.
  \]
  Jensen's inequality for sub-Markovian operators
  (see, e.g., \cite{Haa07}) thus yields
  \[
    j(\alpha Y^\delta) + j^*(\alpha g^\delta) \leq %
    (I+\delta A_1)^{-m} \bigl( j(\alpha Y) + j^*(\alpha g) \bigr),
  \]
  so that
    \[
      \alpha^2 \abs[\big]{Y^\delta g^\delta} \lesssim %
    1 + (I+\delta A_1)^{-m} \bigl( j(\alpha Y) + j^*(\alpha g) \bigr).
    \]
  Since $j(\alpha Y) + j^*(\alpha g) \in L^1((0,T) \times D)$ by
  assumption, the contractivity of the resolvent in $L^1(D)$ and the
  dominated convergence theorem imply that the right-hand side in the
  last inequality is convergent in $L^1((0,T) \times D)$. Hence
  $(Y^\delta g^\delta)_\delta$ is uniformly integrable and, by Vitali's
  theorem,
  \[
    \int_0^t\!\!\int_D g^\delta(s,x) Y^\delta(s,x)\,dx\,ds \longto
    \int_0^t\!\!\int_D g(s,x) Y(s,x)\,dx\,ds
  \]
  for all $t \in [0,T]$. The proof is thus completed.
\end{proof}

As a first important consequence of the generalized It\^o formula we
show that (the first component of) strong solutions are pathwise
strongly continuous in $H$.
\begin{thm}
  \label{thm:cont}
  Let $(X,\xi)$ be the unique strong solution to \eqref{eq:0}
  belonging to $\cJ_2$. Then $X$ has strongly continuous paths in $H$,
  i.e.~there exists $\Omega'\in\cF$ with $\P(\Omega')=1$ such that
  \[
  X(\omega) \in C([0,T]; H) \qquad\forall \omega \in \Omega'.
  \]
\end{thm}
\begin{proof}
  Let $r \in [0,T]$. We have to prove that $X(t) \to X(r)$ in $H$ as
  $t \to r$, $t \in [0,T]$.  It follows from Theorem~\ref{thm:Ito}
  that for every $t\in[0,T]$ there exists $\Omega' \in \cF_0$ with
  $\P(\Omega')=1$ such that
  \begin{align*}
    \frac12\norm{X(t)}^2 - \frac12\norm{X(r)}^2
    &= 
      -\int_r^t \ip{AX(s)}{X(s)}\,ds - \int_r^t\!\!\int_D\xi(s)X(s)\,ds\\
    &+ \frac12 \int_r^t \norm[\big]{B(s)}^2_{\cL^2(U,H)}
      + \int_r^tX(s)B(s,X(s))\,dW(s)
  \end{align*}
  everywhere on $\Omega'$. By the definition of strong solution, we
  can assume that $X \in L^\infty(0,T; H)$, $AX\in L^2(0,T; V')$ and
  $j(X)+j^*(\xi) \in L^1((0,T)\times D)$, as well as that
  $B(\cdot, X) \in L^2(0,T;\cL^2(U,H))$, everywhere on $\Omega'$.
  Since $X\xi = j(X)+j^*(\xi)$, it
  follows that the process
  \[
    [0,T] \ni s \longmapsto \psi(s) := -\ip{AX(s)}{X(s)}
    - \int_D\xi(s)X(s) + \frac12 \norm[\big]{B(s,X(s))}^2_{\cL^2(U,H)}
  \]
  belongs to $L^1(0,T)$ everywhere on $\Omega'$. Therefore, writing
  \[
    \frac12\norm{X(t)}^2 - \frac12\norm{X(r)}^2 = \int_r^t\phi(s)\,ds
    + \int_r^t X(s) B(s,X(s))\,dW(s),
  \]
  since $\psi \in L^1(0,T)$ and the stochastic integral has continuous
  trajectories, we have, as $t \to r$,
  \[
    \norm{X(t)}^2 - \norm{X(r)}^2 \to 0,
  \]
  so that $\norm{X(t)}\to\norm{X(r)}$.  Furthermore, $X(t) \to X(r)$
  weakly in $H$ as $t \to r$ by Theorem~\ref{thm:WP}, hence, since $H$
  is uniformly convex, we conclude that $X(t) \to X(r)$ in $H$ (cf.,
  e.g., \cite[Proposition~3.32]{Bre-FA}).
\end{proof}


\section{Existence and uniqueness}
\label{sec:X0}
We begin with a simple estimate that will be used several times.
\begin{lemma}
  \label{lm:BDGY}
  Let $F$ and $G$ be progressive process with values in $H$ and
  $\cL^2(U,H)$, respectively, such that $FG$ is integrable with
  respect to $W$. For any numbers $p$, $\varepsilon>0$ and any
  stopping time $S$ one has
  \[
    \norm[\big]{\bigl( (FG) \cdot W \bigr)^*_S}_{L^p(\Omega)}
    \lesssim \varepsilon \norm[\big]{F_S^*}^2_{L^{2p}(\Omega)}
      + \frac{1}{\varepsilon} \norm[\big]{%
      G\ind{[\![0,S]\!]}}^2_{L^{2p}(\Omega;L^2(0,T;\cL^2(U,H)))}
  \]
\end{lemma}
\begin{proof}
  The BDG inequality asserts that
  \[
    \norm[\big]{\bigl( (FG) \cdot W \bigr)^*_S}_{L^p(\Omega)}
    \eqsim \norm[\big]{[(FG) \cdot W,(FG) \cdot W]_S^{1/2}}_{L^p(\Omega)},
  \]
  where, by the ideal property of Hilbert-Schmidt operators and
  Young's inequality,
  \begin{align*}
    [(FG) \cdot W,(FG) \cdot W]_S^{1/2}
    &= \biggl( \int_0^S \norm[\big]{F(t)G(t)}_{\cL^2(U,\erre)}^2\,dt
      \biggr)^{1/2}\\
    &\leq \biggl( \int_0^S \norm[\big]{F(t)}^2
      \norm[\big]{G(t)}^2_{\cL^2(U,H)}\,dt \biggr)^{1/2}\\
    &\leq F_S^* \biggl( \int_0^S \norm[\big]{G(t)}^2_{\cL^2(U,H)}\,dt
      \biggr)^{1/2}\\
    &\leq \varepsilon F_S^{*2} + \frac{1}{\varepsilon}
      \int_0^S \norm[\big]{G(t)}^2_{\cL^2(U,H)}\,dt.
  \end{align*}
  Therefore, taking the $L^p(\Omega)$-(quasi)norm on both sides,
  \[
    \norm[\big]{\bigl( (FG) \cdot W \bigr)^*_S}_{L^p(\Omega)}
    \lesssim \varepsilon \norm[\big]{F_S^*}^2_{L^{2p}(\Omega)}
      + \frac{1}{\varepsilon} \norm[\big]{%
        G\ind{[\![0,S]\!]}}^2_{L^{2p}(\Omega;L^2(0,T;\cL^2(U,H)))}
    \qedhere
  \]
\end{proof}

\medskip

Let $(X,\xi)$ and $(Y,\eta) \in \cJ_0$ be strong solutions, in the
sense of Definition~\ref{def:sol}, to the equation
\[
dX + AX\,dt + \beta(X)\,dt \ni B(\cdot,X)\,dW
\]
with initial conditions $X_0$ and $Y_0$, both elements of
$L^0(\Omega,\cF_0,\P;H)$, respectively. Here and throughout this
section we assume that $B$ is locally Lipschitz-continuous in the
sense of assumption (B2).

Let us also introduce the sequence of stopping times
$(T_n)_{n\in\enne}$ defined as
\[
  T_n := \inf \bigl\{ t \geq 0: \norm{X_\Gamma(t)} \geq n \; \text{ or }
  \; \norm{Y_\Gamma(t)} \geq n \bigr\} \wedge T.
\]
Here and in the following, for any $\Gamma \in \cF_0$, we shall denote
multiplication by $\ind{\Gamma}$ by a subscript $\Gamma$. Even though
the stopping times $T_n$ depend on $\Gamma$, we shall not indicate
this explicitly to avoid making the notation too cumbersome.

The stopping times $T_n$ are well defined because, by definition of
$\cJ_0$, $X$ and $Y$ have continuous paths with values in
$H$. Moreover, $T_n \neq 0$ for sufficiently large $n$.

\medskip

The estimate in the following lemma is an essential tool, from which,
for instance, uniqueness and a local property of solutions will follow
as easy corollaries.
\begin{lemma}
  \label{lm:gm}
  Let $\Gamma \in \cF_0$ be such that $X_{0\Gamma},\, Y_{0\Gamma} \in
  L^2(\Omega,\cF_0,\P;H)$. One has, for every $n \in \enne$,
  \[
    \E\bigl( X_\Gamma- Y_\Gamma \bigr)^{*2}_{T_n} \lesssim
    \E\norm[\big]{\pG{X_0}-\pG{Y_0}}^2,
  \]
  with implicit constant depending on $T$ and on the Lipschitz
  constant of $B$ in the ball in $H$ of radius $n$.
\end{lemma}
\begin{proof}
  One has
  \[
    (X-Y) + \int_0^t A(X-Y)\,ds + \int_0^t (\xi-\eta)\,ds = X_0-Y_0
    + \int_0^t (B(X)-B(Y))\,dW.
  \]
  We recall that, for any $\cF_0$-measurable random variable $\zeta$
  and any stochastically integrable process $K$, one has
  $\zeta(K \cdot W) = (\zeta K)\cdot W$. Therefore
  \[
    \pG{(X-Y)} + \int_0^t A\pG{(X-Y)}\,ds
    + \int_0^t \pG{(\xi-\eta)}\,ds = \pG{(X_0-Y_0)}
    + \int_0^t \pG{(B(X)-B(Y))}\,dW.
  \]
  The It\^o formula of Theorem~\ref{thm:Ito} yields
  \begin{align*}
    &\norm[\big]{\pG{X}-\pG{Y}}^2(t \wedge T_n)
    + 2\int_0^{t \wedge T_n}%
      \ip[\big]{A(\pG{X}-\pG{Y})}{\pG{X}-\pG{Y})}\,ds
      + 2\int_0^{t \wedge T_n}\!\!\int_D \pG{((X-Y)(\xi-\eta))}\,ds\\
    &\qquad = \norm[\big]{\pG{X_0} - \pG{Y_0}}^2
      + \int_0^{t \wedge T_n} \norm[\big]{\pG{(B(X)-B(Y))}}^2_{\cL^2(U,H)}\,ds\\
    &\qquad\quad + 2\int_0^{t \wedge T_n} \pG{(X-Y)} \pG{(B(X)-B(Y))}\,dW,
  \end{align*}
  where (a) the second and term terms on the left-hand side are
  positive by monotonicity of $A$ and $\beta$, and by the assumption
  that $\xi \in \beta(X)$, $\eta \in \beta(Y)$ a.e. in
  $\Omega \times (0,T) \times D$; (b) one has
  \[
  \pG{(B(X)-B(Y))} = \ind{\Gamma} \bigl( B(\pG{X}) - B(\pG{Y}) \bigr),
  \]
  hence
  \[
    \ind{[\![0,T_n]\!]} \norm[\big]{\pG{(B(X)-B(Y))}}^2_{\cL^2(U,H)}
    \lesssim_n \ind{[\![0,T_n]\!]} \, \ind{\Gamma}
    \norm[\big]{\pG{X}-\pG{Y}}.
  \]
  Taking supremum in time and expectation,
  \begin{align*}
    \E \bigl( \pG{X}^{T_n} - \pG{Y}^{T_n} \bigr)^{*2}_{t}
    &\lesssim \E\norm[\big]{\pG{X_0}-\pG{Y_0}}^2
    + \int_0^t \E \bigl( \pG{X}^{T_n} - \pG{Y}^{T_n} \bigr)^{*2}_{s}\,ds\\
    &\quad + \E\sup_{s \leq t} \int_0^{s \wedge T_n}%
      \pG{(X-Y)} \pG{(B(X)-B(Y))}\,dW,
  \end{align*}
  where, by Lemma~\ref{lm:BDGY}, the last term on the right-hand side
  is bounded by
  \begin{align*}
    &\varepsilon \E \bigl( \pG{X}^{T_n} - \pG{Y}^{T_n} \bigr)^{*2}_{t}
    + N(\varepsilon) \E\int_0^{t \wedge T_n}%
      \norm[\big]{\pG{(B(X)-B(Y))}}^2_{\cL^2(U,H)}\,ds\\
    &\qquad \leq \varepsilon \E\bigl( \pG{X}^{T_n}-\pG{Y}^{T_n} \bigr)^{*2}_{t}
      + N(\varepsilon,n) \int_0^t \E%
      \bigl( \pG{X}^{T_n} - \pG{Y}^{T_n} \bigr)^{*2}_{s}\,ds.
  \end{align*}
  Choosing $\varepsilon$ small enough, it follows by Gronwall's
  inequality that
  \[
    \E\bigl( X_\Gamma - Y_\Gamma \bigr)^{*2}_{T_n} =
    \E\bigl( \pG{X}^{T_n}-\pG{Y}^{T_n} \bigr)^{*2}_{T} \lesssim
    \E\norm[\big]{\pG{X_0}-\pG{Y_0}}^2,
  \]
  with an implicit constant that depends on $T$ and on the Lipschitz
  constant of $B$ on the ball in $H$ of radius $n$.
\end{proof}

\begin{coroll}
  \label{cor:uniq}
  Uniqueness of strong solutions in $\cJ_0$ holds for \eqref{eq:0}.
\end{coroll}
\begin{proof}
  Let $(X,\xi)$, $(Y,\eta) \in \cJ_0$ be strong solutions to
  \eqref{eq:0}. For any $\Gamma \in \cF_0$ such that
  $X_{0\Gamma} \in L^2(\Omega;H)$ the previous lemma yields
  $X^{T_n}_\Gamma = Y^{T_n}_\Gamma$ for all $n \in \enne$, hence
  $X_\Gamma = Y_\Gamma$.  Writing
  \[
  \Omega = \bigcup_{k\in\enne} \Omega_k, \qquad
  \Omega_k := \bigl\{ \omega \in \Omega: \norm{X_0(\omega)} \leq k \bigr\},
  \]
  and choosing $\Gamma$ as $\Omega_k$, it follows that
  $X\ind{\Omega_k} = Y\ind{\Omega_k}$ for all $k$, hence $X=Y$.  By
  comparison, $\xi = \eta$ a.e. in $\Omega \times (0,T) \times D$.
\end{proof}
\begin{rmk}
  To prove the corollary, by inspection of the proof of
  Lemma~\ref{lm:gm} it is evident that one may directly take
  $\Gamma=\Omega$, as in this case $X_0-Y_0=0$, whose second moment is
  obviously finite. This immediately implies $X^{T_n} = Y^{T_n}$ for
  all $n \in \enne$, hence $X=Y$.
\end{rmk}

\begin{coroll}
  \label{cor:loc}
  Let $\Gamma \in \cF_0$. If $X_{0\Gamma}=Y_{0\Gamma}$, then
  $X_\Gamma=Y_\Gamma$, and $\xi_\Gamma = \eta_\Gamma$ a.e. in $\Omega
  \times (0,T) \times D$.
\end{coroll}
\begin{proof}
  Write $\Omega = \bigcup_{k\in\enne} \Omega_k$, where
  \[
  \Omega_k := \bigl\{ \omega \in \Omega: \norm{X_0(\omega)} \leq k \bigr\}
  \cap \bigl\{ \omega \in \Omega: \norm{Y_0(\omega)} \leq k \bigr\}.
  \]
  Then $X_0\ind{\Gamma \cap \Omega_k}$, $Y_0\ind{\Gamma \cap \Omega_k}
  \in L^2(\Omega;H)$, and Lemma~\ref{lm:gm} implies that
  $X_{\Gamma\cap\Omega_k} = Y_{\Gamma\cap\Omega_k}$ for all $k \in
  \enne$, hence $X_\Gamma = Y_\Gamma$, as well as, again by
  comparison, $\xi_\Gamma = \eta_\Gamma$ a.e. in $\Omega \times (0,T)
  \times D$.
\end{proof}

Now that uniqueness is cleared, we turn to the question of existence
of strong solutions.  For this we need some preparations. For $R>0$,
let us consider the truncation operator $\sigma_R:H \to H$ defined as
\[
\sigma_R: x \longmapsto
\begin{cases}
  x, & \norm{x} \leq R,\\
  R  x / \norm{x}, & \norm{x} > R.
\end{cases}
\]
We shall then define
\begin{align*}
  B_R: \Omega \times[0,T] \times H &\longto \cL^2(U,H)\\
  (\omega,t,x) &\longmapsto B(\omega,t,\sigma_R(x)).
\end{align*}
Let us check that $B_R$ is Lipschitz-continuous for every $R>0$.  The
progressive measurability of $B_R$ follows from the one of $B$ and the
fact that $\sigma_R:H\to H$ is (Lipschitz) continuous. Moreover, since
$\sigma_R$ is $1$-Lipschitz continuous, thanks to the local Lipschitz
continuity and the linear growth of $B$, for every $\omega\in\Omega$,
$t\in[0,T]$ and $x,y\in H$ one has
\[
  \norm{B_R(\omega,t,x)-B_R(\omega,t,y)}_{\cL^2(U,H)}\leq
  N_R\norm{\sigma_R(x)-\sigma_R(y)}\leq N_R\norm{x-y}
\]
as well as
\[
  \norm{B_R(\omega, t, x)}_{\cL^2(U,H)}\leq
  N(1+\norm{\sigma_R(x)})\leq N(1+\norm{x}).
\]

Thanks to Theorems~\ref{thm:WP} and \ref{thm:cont}, as well as
Lemma~\ref{lm:gm}, the equation
\begin{equation}
  \label{eq:n}
  dX_n + AX_n\,dt + \beta(X_n)\,dt = B_n(X_n)\,dW, \qquad X_n(0)=X_0,
\end{equation}
admits a strong solution $(X_n,\xi_n)$, which belongs to $\cJ_2$ and
is unique in $\cJ_0$, for every $n \in \enne$.\footnote{Note that
  Theorem~\ref{thm:WP} only shows that $(X_n,\xi_n)$ is unique in
  $\cJ_2$, while Lemma~\ref{lm:gm} yields uniqueness in the larger
  space $\cJ_0$.}  Moreover, by the strong continuity of the paths of
$X_n$, one can define the increasing sequence of stopping times
$(\tau_n)_{n\in\enne}$ by
\[
\tau_n := \inf \bigl\{ t \in [0,T]: \norm{X_n(t)} \geq n \bigr\},
\]
as well as the stopping time
\[
\tau := \lim_{n\to\infty} \tau_n = \sup_{n\in\enne} \tau_n.
\]

As first step we show that the sequence of processes $(X_n,\xi_n)$ satisfies a
sort of consistency condition.
\begin{lemma}
  \label{lm:ind}
  One has $X_{n+1}^{\tau_n} = X_n^{\tau_n}$ for all $n \in \enne$, as
  well as $\xi_n\ind{[\![0,\tau_n]\!]} = \xi_{n+1}\ind{[\![0,\tau_n]\!]}$ in
  $L^0(\Omega \times (0,T) \times D)$.
\end{lemma}
\begin{proof}
  It\^o's formula yields, in view of the monotonicity of $A$ and
  $\beta$,
  \begin{align*}
  \norm[\big]{X_{n+1}-X_n}^2(t \wedge \tau_n)
  &\lesssim \int_0^{t \wedge \tau_n} (X_{n+1}-X_n)%
  \bigl( B_{n+1}(X_{n+1})-B_n(X_n) \bigr)(s)\,dW(s)\\
  &\quad 
  + \int_0^{t \wedge \tau_n}
  \norm[\big]{B_{n+1}(X_{n+1}(s))-B_n(X_n(s))}^2_{\cL^2(U,H)}\,ds.
  \end{align*}
  Note that $B_{n+1}=B_n$ on the ball of radius $n$ in $H$, hence
  $B_n(X_n)=B_{n+1}(X_n)$ on $[\![0,\tau_n]\!]$. Therefore, since
  $B_{n+1}$ is Lipschitz continuous,
  \begin{align*}
  \E \bigl( X^{\tau_n}_{n+1} - X^{\tau_n}_n \bigr)^{*2}_t
  &\lesssim
  \E\bigl(
  ((X_{n+1}-X_n) (B_{n+1}(X_{n+1}) - B_{n+1}(X_n))) \cdot W \bigr)^*_{t \wedge \tau_n}\\
  &\quad + \int_0^t \E\bigl( X^{\tau_n}_{n+1} - X^{\tau_n}_n \bigr)^{*2}_s\,ds,
  \end{align*}
  where the first term on the right-hand side can be estimated, thanks
  to the BDG inequality and the ideal property of Hilbert-Schmidt
  operators, by
  \begin{align*}
  &\E\biggl( \int_0^{t \wedge \tau_n}
  \norm[\big]{X_{n+1}-X_n}^2(s)
  \norm[\big]{B_{n+1}(X_{n+1}(s))-B_{n+1}(X_n(s))}^2_{\cL^2(U,H)}\,ds
  \biggr)^{1/2}\\
  &\hspace{3em} \lesssim_n \E\bigl(%
   X_{n+1}^{\tau_n} - X_n^{\tau_n} \bigr)^*_t \biggl( \int_0^t
  \norm[\big]{X_{n+1}^{\tau_n}-X_n^{\tau_n}}^2(s) \biggr)^{1/2}\\
  &\hspace{3em} \lesssim_n
  \varepsilon \E \bigl( X_{n+1}^{\tau_n} - X_n^{\tau_n} \bigr)^{*2}_t
  + \frac{1}{\varepsilon} \int_0^t%
  \E \bigl( X_{n+1}^{\tau_n} - X_n^{\tau_n} \bigr)^{*2}_s\,ds.
  \end{align*}
  Choosing $\varepsilon$ small enough, Gronwall's inequality implies
  \[
  \E \bigl( X_{n+1}^{\tau_n} - X_n^{\tau_n} \bigr)^{*2}_t = 0
  \]
  for all $t \leq T$, hence $X_{n+1}^{\tau_n} =
  X_n^{\tau_n}$. The first claim is thus proved.

  In order to prove the second claim, note that it holds
  \begin{align*}
    X_{n+1}^{\tau_n}(t) + \int_0^{t \wedge \tau_n} AX_{n+1}\,ds
    + \int_0^{t \wedge \tau_n} \xi_{n+1}\,ds
    &= X_0 + \int_0^{t \wedge \tau_n} B_{n+1}(X_{n+1})\,dW,\\
    X_{n}^{\tau_n}(t) + \int_0^{t \wedge \tau_n} AX_{n}\,ds
    + \int_0^{t \wedge \tau_n} \xi_{n}\,ds
    &= X_0 + \int_0^{t \wedge \tau_n} B_{n}(X_{n})\,dW,
  \end{align*}
  where $B_n(X_n)$ on the right-hand side of the second identity can
  be replaced by $B_{n+1}(X_{n+1})$ because the paths of
  $X^{\tau_n}_{n+1}$ remain within a ball of radius $n$ in $H$ and
  $X_{n+1}^{\tau_n}=X_n^{\tau_n}$. This identity also yields, by
  comparison,
  \[
    \int_0^t \xi_{n+1}\ind{[\![0,\tau_n]\!]}\,ds
    = \int_0^{t \wedge \tau_n} \xi_{n+1}\,ds
    = \int_0^{t \wedge \tau_n} \xi_{n}\,ds
    = \int_0^t \xi_{n}\ind{[\![0,\tau_n]\!]}\,ds,
  \]
  which implies the second claim.\footnote{The argument in fact proves
    the following slightly stronger statement: setting
    $\Xi_n:=\int_0^\cdot \xi_n\,ds$, the processes
    $\Xi_{n+1}^{\tau_n}$ and $\Xi_n^{\tau_n}$ are indistinguishable
    for all $n$.}
\end{proof}
The lemma implies that one can define processes $X$ and $\xi$ on
$[\![0,\tau]\!]$ by the prescriptions $X := X_n$ and $\xi := \xi_n$ on
$[\![0,\tau_n]\!]$ for all $n \in \enne$, or equivalently (but perhaps
less tellingly), as $X=\lim_{n\to\infty} X_n$ and
$\xi=\lim_{n\to\infty} \xi_n$.

\medskip

We are now going to show that the linear growth assumption on $B$
implies that $\tau=T$. We shall first establish a priori estimates for
the solution to equation \eqref{eq:n}.
\begin{lemma}
  \label{lm:ap}
  There exists a constant $N>0$, independent of $n$, such that
  \[
    \E\norm[\big]{X_n}^2_{C([0,T];H)}
    + \E\norm[\big]{X_n}^2_{L^2(0,T;V)} 
    + \E\norm[\big]{\xi_nX_n}_{L^1((0,T) \times D)}
    < N \bigl( 1 + \E\norm{X_0}^2 \bigr).
  \]
\end{lemma}
\begin{proof}
  The It\^o formula of Theorem~\ref{thm:Ito} yields
  \begin{align*}
    &\norm{X_n(t)}^2
      + 2\int_0^{t} \ip{AX_n(s)}{X_n(s)}\,ds
      + 2\int_0^{t} \!\!\int_D\xi_n(s)X_n(s)\,dx\,ds\\
    &\hspace{3em} = \norm{X_0}^2
      + \int_0^{t} \norm[\big]{B_n(s,X_n(s))}^2_{\cL^2(U,H)}\,ds
      + 2\int_0^{t} X_n(s)B_n(s,X_n(s))\,dW(s),
  \end{align*}
  where, recalling that $B_n=B(\cdot,\cdot,\sigma_n(\cdot))$ and
  $\sigma_n$ is a contraction in $H$, and that $B$ grows at most
  linearly,
  \[
  \int_0^t \norm[\big]{B_n(s,X_n(s))}^2_{\cL^2(U,H)}
  \lesssim T + \int_0^T \norm{X_n(s)}^2\,ds.
  \]
  Denoting the stochastic integral on the right-hand side by $M_n$,
  taking supremum in time and expectation we get, by the coercivity of
  $A$,
  \begin{align*}
    &\E\norm[\big]{X_n}^2_{C([0,T];H)}
    + \E\norm[\big]{X_n}^2_{L^2(0,T;V)} 
    + \E\norm[\big]{\xi_nX_n}_{L^1((0,T) \times D)}\\
    &\hspace{3em} \lesssim 1 + \E\norm{X_0}^2
      + \E\int_0^T \norm{X_n(s)}^2\,ds
      + \E M_T^{*2},
  \end{align*}
  where the implicit constant depends on $T$. By Lemma~\ref{lm:BDGY}
  we have, for any $\varepsilon>0$,
  \[
    \E M_T^{*2} \lesssim \varepsilon \E\norm[\big]{X_n}^2_{C([0,T];H)}
    + N(\varepsilon) \E \int_0^T \norm{X_n(s)}^2\,ds,
  \]
  therefore, choosing $\varepsilon$ sufficiently small,
  \begin{align*}
    &\E\norm[\big]{X_n}^2_{C([0,T];H)}
    + \E\norm[\big]{X_n}^2_{L^2(0,T;V)} 
    + \E\norm[\big]{\xi_nX_n}_{L^1((0,T) \times D)}\\
    &\hspace{3em} \lesssim 1 + \E\norm{X_0}^2
      + \E\int_0^T \norm{X_n(s)}^2\,ds.
  \end{align*}
  Since this inequality holds also with $T$ replaced by any
  $t \in ]0,T]$, we also have
  \[
    \E\norm[\big]{X_n}^2_{C([0,t];H)} \lesssim
    1 + \E\norm{X_0}^2 + \int_0^T \E\norm[\big]{X_n}^2_{C([0,s];H)}\,ds,
  \]
  hence, by Gronwall's inequality,
  \[
  \E\norm[\big]{X_n}^2_{C([0,T];H)} \lesssim 1 + \E\norm{X_0}^2,
  \]
  with implicit constant depending on $T$. Since $C([0,T];H)$ is
  continuously embedded in $L^2(0,T;H)$, one easily deduces
  \[
    \E\norm[\big]{X_n}^2_{C([0,T];H)} +
    \E\norm[\big]{X_n}^2_{L^2(0,T;V)} +
    \E\norm[\big]{\xi_nX_n}_{L^1((0,T) \times D)} \lesssim 1 +
    \E\norm{X_0}^2.
  \qedhere
  \]
\end{proof}

\begin{lemma}
  \label{lm:tau}
  One has
  \[
    \P\Bigl( \limsup_{n \to \infty} \{\tau_n \leq T\} \Bigr) = 0.
  \]
  In particular, $\tau=T$.
\end{lemma}
\begin{proof}
  By Markov's inequality and the previous lemma,
  \[
    \P\bigl( \norm{X_n}_{C([0,T];H)} \geq n \bigr)
    \leq \frac{1}{n^2} \E\norm{X_n}^2_{C([0,T]; H)}
    \lesssim \frac{1}{n^2} \bigl( 1 + \E\norm{X_0}^2).
  \]
  Since the event $\{ \norm{X_n}_{C([0,T];H)} \geq n \}$ coincides
  with $\{\tau_n \leq T\}$, one has
  \[
  \sum_{n=1}^\infty \P\bigl( \tau_n \leq T \bigr) < \infty,
  \]
  thus also, by the Borel-Cantelli lemma,
  \[
  \P\left(\bigcap_{n\in\enne}\bigcup_{k\geq n}\{\tau_k\leq T\}\right) = 0.
  \]
  In other words, the sequence $(\tau_n)$ is ultimately constant: for
  each $\omega$ in a subset of $\Omega$ of $\P$-measure one, there
  exists $m=m(\omega)$ such that $\tau_n(\omega)=T$ for all $n>m$. In
  particular, $\tau=T$ $\P$-almost surely.
\end{proof}

This lemma implies that the processes $X$ and $\xi$ defined
immediately after the proof of Lemma~\ref{lm:ind} are indeed defined
on the whole interval $[0,T]$. 

We can now prove the first existence result.
\begin{thm} 
  \label{thm:int}
  Assume that $X_0 \in L^2(\Omega,\cF_0,\P;H)$. Then equation
  \eqref{eq:0} admits a unique strong solution, which belongs to
  $\mathscr{J}_2$.
\end{thm}
\begin{proof}
  Uniqueness of strong solutions is proved, in more generality, by
  Corollary~\ref{cor:uniq}. Let us prove existence.  By stopping at
  $\tau_n$, one has
  \[
    X_n^{\tau_n}(t) + \int_0^{t \wedge \tau_n} AX_n(s)\,ds
    + \int_0^{t \wedge \tau_n} \xi_n(s)\,ds
    = X_0 + \int_0^{t \wedge \tau_n} B_n(X_n(s))\,ds,
  \]
  where, by definition of $X$, $X_n^{\tau_n}=X^{\tau_n}$, as well as,
  by definition of $B_n$,
  \[
    \int_0^{t \wedge \tau_n} B_n(X_n(s))\,ds
    = \int_0^{t \wedge \tau_n} B(X(s))\,ds.
  \]
  Similarly, by definition of $\xi$ it follows that
  \[
    \int_0^{t \wedge \tau_n} \xi_n(s)\,ds
    = \int_0^{t \wedge \tau_n} \xi(s)\,ds,
  \]
  hence that
  \[
    X^{\tau_n}(t) + \int_0^{t \wedge \tau_n} AX(s)\,ds
    + \int_0^{t \wedge \tau_n} \xi(s)\,ds
    = X_0 + \int_0^{t \wedge \tau_n} B(X(s))\,ds.
  \]
  Since this identity holds for all $n \in \enne$ and $\tau_n \to T$
  as $n \to \infty$, we infer that
  \[
    X(t) + \int_0^t AX(s)\,ds
    + \int_0^t \xi(s)\,ds
    = X_0 + \int_0^t B(X(s))\,ds
  \]  
  for all $t \in [0,T]$ $\P$-a.s.. Moreover, for every $n \in \enne$,
  $\xi_n \in \beta(X_n)$ a.e. in $\Omega \times (0,T) \times D$, hence
  $\xi_n\ind{[\![0,\tau_n]\!]} \in \beta(X_n)\ind{[\![0,\tau_n]\!]}$,
  thus also
  $\xi\ind{[\![0,\tau_n]\!]} \in \beta(X)\ind{[\![0,\tau_n]\!]}$
  a.e. in $\Omega \times (0,T) \times D$. Recalling that
  $\tau_n \to T$ as $n \to \infty$, this in turn implies
  $\xi \in \beta(X)$ a.e. in $\Omega \times (0,T) \times D$.

  Moreover, since $(X,\xi)$ is the almost sure limit of $(X_n,\xi_n)$,
  we immediately infer that $X$ and $\xi$ are predictable $H$-valued
  and $L^1(D)$-valued processes, respectively. The a priori estimates
  of Lemma~\ref{lm:ap} and Fatou's lemma then yield
  \[
    X \in L^2(\Omega;C([0,T];H)) \cap L^2(\Omega;L^2(0,T;V)),
    \qquad \xi \in L^1(\Omega \times (0,T) \times D).
  \]
  Similarly, $\xi_n \in \beta(X_n)$ implies
  $X_n\xi_n = j(X_n)+j^*(\xi_n)$, hence
  \[
    \E\int_0^T\!\!\int_D \bigl( j(X_n)+j^*(\xi_n) \bigr) \lesssim 1 +
    \E\norm{X_0}^2
  \]
  for all $n \in \enne$, and again by Fatou's lemma, as well as by the
  lower-semicontinuity of convex integrals, one obtains
  \[
    j(X) + j^*(\xi) \in L^1(\Omega;L^1(0,T;L^1(D))).
  \]
  We have thus proved that $(X,\xi) \in \cJ_2$, so the proof is
  completed.
\end{proof}

The second existence result, which allows $X_0$ to be merely
$\cF_0$-measurable, follows by a further ``gluing'' procedure.
\begin{thm}
  Assume that $X_0 \in L^0(\Omega,\cF_0,\P;H)$. Then equation
  \eqref{eq:0} admits a unique strong solution.
\end{thm}
\begin{proof}
  Uniqueness of strong solutions has already been proved in
  Corollary~\ref{cor:uniq}. It is hence enough to prove existence. Let
  us define the sequence $(\Gamma_n)_{n\in\enne}$ of elements of
  $\cF_0$ as
  \[
  \Gamma_n := \bigl\{ \omega \in \Omega: \, \norm{X_0} \leq n \bigr\}.
  \]
  It is evident that $(\Gamma_n)$ is a sequence increasing to
  $\Omega$, and that
  $X_{0\Gamma_n} = X_0\ind{\Gamma_n} \in L^2(\Omega;H)$. Therefore, by
  the previous theorem, for each $n \in \enne$ there exists a unique
  strong solution $(X_n,\xi_n)$ to \eqref{eq:0} with initial condition
  $X_{0\Gamma_n}$. By the local property of solutions established in
  Corollary~\ref{cor:loc}, we have that $X_{n+1}\ind{\Gamma_n}$ and
  $X_n\ind{\Gamma_n}$ are indistinguishable, and
  $\xi_{n+1}\ind{\Gamma_n}=\xi_n\ind{\Gamma_n}$ a.e. in
  $\Omega \times (0,T) \times D$. Since $(\Gamma_n)$ is increasing, it
  makes sense to define the processes $X$ and $\xi$ by
  \[
    X\ind{\Gamma_n} = X_n\ind{\Gamma_n}, \quad
    \xi\ind{\Gamma_n} = \xi_n\ind{\Gamma_n}
  \]
  for all $n \in \enne$. This amounts to saying that $X$ and $\xi$
  are the $\P$-a.s. limits of $X_n$ and $\xi_n$, respectively,
  which
  immediately implies that $X$ and $\xi$ are predictable processes
  with values in $H$ and $L^1(D)$, respectively. Moreover, by
  construction, we also have
  \[
    X \in L^0(\Omega;C([0,T];H) \cap L^2(0,T;V)), \qquad
    \xi \in L^0(\Omega;L^1(0,T;L^1(D)))
  \]
  In fact, writing $E:=C([0,T];H) \cap L^2(0,T;V)$ for compactness of
  notation, by the previous theorem we have $X_n \in L^2(\Omega;E)$
  and $\xi \in L^1(\Omega \times (0,T) \times D)$, and for any
  arbitrary but fixed $\omega$ in a subset of $\Omega$ of probability
  one, there exists $n=n(\omega)$ such that
  $(X(\omega),\xi(\omega)) =(X_n(\omega),\xi_n(\omega)) \in E \times
  L^1((0,T) \times D)$. Furthermore, since $\xi_n \in \beta(X_n)$
  a.e. for all $n \in \enne$, it is easy to see that
  \[
    \xi\ind{\Gamma_n} = \xi_n\ind{\Gamma_n} \in \beta(X_n)\ind{\Gamma_n}
    = \beta(X_n\ind{\Gamma_n}) \ind{\Gamma_n}
    = \beta(X) \ind{\Gamma_n}
  \]
  for all $n \in \enne$, so that $\xi \in \beta(X)$ a.e. because
  $\Gamma_n \uparrow \Omega$. Similarly,
  \[
    j(X_n) \ind{\Gamma_n} = j(\ind{\Gamma_n}X_n) \ind{\Gamma_n}
    = j(X) \ind{\Gamma_n}
  \]
  as well as, by the same reasoning,
  $j^*(\xi_n) \ind{\Gamma_n} = j^*(\xi) \ind{\Gamma_n}$. Since, by the
  previous theorem,
  $j(X_n) + j^*(\xi_n) \in L^1(\Omega;L^1((0,T) \times D)$ for all
  $n \in \enne$, it follows that
  \[
    \bigl( j(X) + j^*(\xi) \bigr) \ind{\Gamma_n} \in
    L^1(\Omega;L^1((0,T) \times D) \qquad \forall n \in \enne,
  \]
  hence $j(X) + j^*(\xi) \in L^0(\Omega;L^1((0,T) \times D)$.
\end{proof}


\section{Moment estimates and dependence on the initial datum}
\label{sec:mom}
We are now going to show that the integrability of the solution is
determined by the integrability of the initial condition.
\begin{thm}
  Let $p \geq 0$. If $X_0 \in L^p(\Omega,\cF_0,\P;H)$, then the unique
  strong solution to equation \eqref{eq:0} belongs to
  $\mathscr{J}_p$.
\end{thm}
\begin{proof}
  It\^o's formula yields
  \begin{align*}
    &\norm{X(t)}^2
      + 2\int_0^{t} \ip{AX(s)}{X(s)}\,ds
      + 2\int_0^{t} \!\!\int_D\xi(s)X(s)\,dx\,ds\\
    &\hspace{3em} = \norm{X_0}^2
      + \int_0^{t} \norm[\big]{B(s,X(s))}^2_{\cL^2(U,H)}\,ds
      + 2\int_0^{t} X(s)B(s,X(s))\,dW(s).
  \end{align*}
  For any $\alpha>0$, it follows by the integration-by-parts formula
  that
  \begin{align*}
  &e^{-2\alpha t} \norm{X(t)}^2
    + 2\alpha \int_0^t e^{-2\alpha s} \norm{X(s)}^2\,ds
    + 2\int_0^t e^{-2\alpha s} \ip{AX(s)}{X(s)}\,ds\\
  &\quad + 2\int_0^t\!\!\int_D e^{-2\alpha s} \xi(s)X(s)\,dx\,ds\\
  &\quad \hspace{3em} = \norm{X_0}^2
    + \int_0^t e^{-2\alpha s} \norm[\big]{B(s,X(s))}^2_{\cL^2(U,H)}\,ds
    + 2\int_0^t e^{-2\alpha s} X(s)B(s, X(s))\,dW(s).
  \end{align*}
  Let $M$ denote the stochastic integral on the right-hand side, and
  $Y(t):=e^{-\alpha t}X(t)$. Since $X$ has continuous paths in $H$,
  one can introduce the sequence of stopping times
  $(T_n)_{n\in\enne}$, increasing to $T$, as
  \[
  T_n := \inf \bigl\{ t \geq 0: \, \norm{X(t)} \geq n \bigr\} \wedge T.
  \]
  It follows by the local Lipschitz-continuity property of $B$ that
  \begin{align*}
  &\norm{Y^{T_n}(t)}^2
    + 2\alpha \int_0^{t \wedge T_n} \norm{Y(s)}^2\,ds
    + 2C \int_0^{t  \wedge T_n}  \norm{Y(s)}^2_V\,ds
    +2\int_0^{t \wedge T_n}\!\!\int_D e^{-2\alpha s} \xi(s)X(s)\,dx\,ds\\
  &\hspace{3em} \leq \norm{X_0}^2
    + \int_0^{t \wedge T_n} e^{-2\alpha s}
    \norm[\big]{B_n(s,X(s))}^2_{\cL^2(U,H)}\,ds
    + 2M^{T_n}(t).
  \end{align*}
  Recalling that $B_n=B(\cdot,\cdot,\sigma_n(\cdot))$ and
  $\sigma_n$ is a contraction in $H$, and that $B$ grows at most
  linearly, one has
  \[
  e^{-2\alpha s} \norm[\big]{B_n(s,X(s))}^2_{\cL^2(U,H)}
  \lesssim e^{-2\alpha s} + \norm{Y(s)}^2,
  \]
  hence
  \begin{equation}
    \label{eq:alfetta}
    \int_0^{t \wedge T_n} e^{-2\alpha s}
    \norm[\big]{B_n(s,X(s))}^2_{\cL^2(U,H)}\,ds \lesssim
  \frac{1}{2\alpha} + \int_0^{t \wedge T_n} \norm{Y(s)}^2\,ds.
  \end{equation}
  Taking supremum in time and the $L^{p/2}(\Omega)$-(quasi)norm,
  recalling the BDG inequality and the fact that
  $e^{-\alpha t}\xi_nX_n\geq e^{-\alpha T}\xi_nX_n$, we are left with
  \begin{align*}
  &\norm[\big]{Y_{T_n}^*}^2_{L^p(\Omega)}
    + \alpha \norm[\big]{Y\ind{[\![0,T_n]\!]}}^2_{L^p(\Omega;L^2(0,T;H))}
  + \norm[\big]{Y\ind{[\![0,T_n]\!]}}^2_{L^p(\Omega;L^2(0,T;V))}\\
  &\qquad +e^{-\alpha T}
    \norm[\big]{\xi X\ind{[\![0,T_n]\!]}}_{L^{p/2}(\Omega; L^1((0,T)\times D))}\\
  &\qquad\qquad \lesssim \norm[\big]{X_0}^2_{L^p(\Omega;H)}
    + \frac{1}{2\alpha}
    + \norm[\big]{Y\ind{[\![0,T_n]\!]}}^2_{L^p(\Omega;L^2(0,T;H))}
    + \norm[\big]{[M,M]_{T_n}^{1/2}}_{L^{p/2}(\Omega)}.
  \end{align*}
  Lemma~\ref{lm:BDGY} and \eqref{eq:alfetta} yield
  \[
  [M,M]^{1/2}_{T_n}
  \lesssim \varepsilon Y_{T_n}^{*2}
    + \frac{1}{\varepsilon} \biggl( \frac{1}{2\alpha}
    + \norm[\big]{Y\ind{[\![0,T_n]\!]}}^2_{L^2(0,T;H)} \biggr),
  \]
  hence
  \[
  \norm[\big]{[M,M]_{T_n}^{1/2}}_{L^{p/2}(\Omega)}
  \lesssim \varepsilon \norm[\big]{Y_{T_n}^*}^2_{L^p(\Omega)}
  + \frac{1}{\varepsilon}%
  \norm[\big]{Y\ind{[\![0,T_n]\!]}}^2_{L^p(\Omega;L^2(0,T;H))}
  + \frac{1}{2\alpha\varepsilon},
  \]
  where the implicit constant is independent of $\alpha$ and of an
  arbitrary $\varepsilon>0$ to be chosen later. We thus have
  \begin{align*}
  &\norm[\big]{Y_{T_n}^*}^2_{L^p(\Omega)}
    + \alpha \norm[\big]{Y\ind{[\![0,T_n]\!]}}^2_{L^p(\Omega;L^2(0,T;H))}
  + \norm[\big]{Y\ind{[\![0,T_n]\!]}}^2_{L^p(\Omega;L^2(0,T;V))}\\
    &\qquad +e^{-\alpha T}
      \norm[\big]{\xi X\ind{[\![0,T_n]\!]}}_{L^{p/2}(\Omega; L^1((0,T)\times D))}\\
  &\qquad\qquad \lesssim \norm[\big]{X_0}^2_{L^p(\Omega;H)}
    + \varepsilon \norm[\big]{(Y^{T_n})^*_T}^2_{L^p(\Omega)}
    + (1+1/\varepsilon) \norm[\big]{Y\ind{[\![0,T_n]\!]}}^2_{L^p(\Omega;L^2(0,T;H))}
    + \frac{1}{2\alpha} (1+1/\varepsilon).
\end{align*}
Since the implicit constant is independent of $\alpha$ and
$\varepsilon$, one can take $\varepsilon$ small enough and $\alpha$
large enough so that
\[
  \norm[\big]{Y_{T_n}^*}^2_{L^p(\Omega)}
  + \norm[\big]{Y\ind{[\![0,T_n]\!]}}^2_{L^p(\Omega;L^2(0,T;V))}
  +\norm[\big]{\xi X\ind{[\![0,T_n]\!]}}_{L^{p/2}(\Omega; L^1((0,T)\times D))}
  \lesssim 1 + \norm[\big]{X_0}^2_{L^p(\Omega;H)}.
\]
As the implicit constant is independent of $n$ and $T_n$ increases to $T$,
we get
\[
  \norm[\big]{Y}^2_{L^p(\Omega;C([0,T];H))}
  + \norm[\big]{Y}^2_{L^p(\Omega;L^2(0,T;V))}
  +\norm[\big]{\xi X}_{L^{p/2}(\Omega; L^1((0,T)\times D))}
  \lesssim 1 + \norm[\big]{X_0}^2_{L^p(\Omega;H)}.
\]
The proof is completed noting that, for
$E:=C([0,T];H) \cap L^2(0,T;V)$,
\[
  \norm{X}_E \leq e^{\alpha T} \norm{Y}_E.
  \qedhere
\]
\end{proof}

If $B$ is Lipschitz-continuous, related arguments show that the
solution map is Lipschitz-continuous between spaces with finite $p$-th
moment in the whole range $p \in [0,\infty[$. We consider the cases
$p>0$ and $p=0$ separately.
\begin{prop}
  Let $p>0$. If $B$ is Lipschitz-continuous in the sense of assumption
  \emph{(B1)}, then the solution map
  \begin{align*}
    L^p(\Omega;H) &\longto L^p(\Omega;C([0,T];H)) \cap
                    L^p(\Omega;L^2(0,T;V))\\
    X_0 &\longmapsto X
  \end{align*}
  is Lipschitz-continuous.
\end{prop}
\begin{proof}
  Let $X_0$, $Y_0 \in L^p(\Omega;H)$. The previous theorem asserts
  that the (unique) strong solutions $(X,\xi)$ and $(Y,\eta)$ to
  \eqref{eq:0} with initial condition $X_0$ and $Y_0$, respectively,
  belong to $L^p(\Omega;E)$, where, as before, $E$ stands for
  $C([0,T];H) \cap L^2(0,T;V)$. By It\^o's formula,
  \begin{align*}
    &\norm{X-Y}^2
      + 2\int_0^t \ip{A(X-Y)}{X-Y}\,ds
    + 2\int_0^{t} \!\!\int_D(\xi-\eta)(X-Y)\,ds\\
    &\hspace{3em} = \norm{X_0-Y_0}^2
      + \int_0^{t} \norm[\big]{B(X)-B(Y)}^2_{\cL^2(U,H)}\,ds\\
    &\hspace{3em}\quad
      + 2\int_0^{t} (X-Y)(B(X)-B(Y))\,dW,
  \end{align*}
  where the third term on the left-hand side is positive by
  monotonicity of $\beta$. Let $\alpha>0$ be a constant to be chosen
  later, and set $X_\alpha:= X e^{-\alpha \cdot}$,
  $Y_\alpha:= Y e^{-\alpha \cdot}$. It follows by the
  integration-by-parts formula, in complete analogy to the proof of
  the previous theorem, by the Lipschitz continuity of $B$, and by the
  coercivity of $A$, that
  \begin{align*}
  &\norm{X_\alpha - Y_\alpha}^2
    + \alpha \int_0^t \norm{X_\alpha-Y_\alpha}^2\,ds
    + \int_0^t  \norm{X_\alpha-Y_\alpha}^2_V\,ds\\
  &\hspace{3em} \lesssim \norm{X_0 - Y_0}^2
    + \int_0^t \norm[\big]{X_\alpha-Y_\alpha}^2\,ds
    + M,
  \end{align*}
  where $M:= \bigl( e^{-2\alpha \cdot}(X-Y)(B(X)-B(Y)) \bigr) \cdot W$.
  Taking supremum in time and the $L^{p/2}(\Omega)$-(quasi)norm yields
  \begin{align*}
  &\norm[\big]{X_\alpha-Y_\alpha}^2_{L^p(\Omega;C([0,T];H))}
    + \alpha \norm[\big]{X_\alpha-Y_\alpha}^2_{L^p(\Omega;L^2(0,T;H))}
  + \norm[\big]{X_\alpha-Y_\alpha}^2_{L^p(\Omega;L^2(0,T;V))}\\
  &\hspace{3em} \lesssim \norm[\big]{X_0-Y_0}^2_{L^p(\Omega;H)}
    + \norm[\big]{X_\alpha - Y_\alpha}^2_{L^p(\Omega;L^2(0,T;H))}
    + \norm[\big]{M^*_{T}}_{L^{p/2}(\Omega)},
  \end{align*}
  where, by Lemma~\ref{lm:BDGY},
  \[
  \norm[\big]{M^*_{T}}_{L^{p/2}(\Omega)} \lesssim
  \varepsilon \norm[\big]{X_\alpha-Y_\alpha}^2_{L^p(\Omega;C([0,T];H))}
  + N(\varepsilon) \norm[\big]{X_\alpha - Y_\alpha}^2_{L^p(\Omega;L^2(0,T;H))}
  \]
  for any $\varepsilon>0$. Choosing first $\varepsilon$ small enough,
  then $\alpha$ sufficiently large, we obtain
  \[
    \norm[\big]{X_\alpha-Y_\alpha}^2_{L^p(\Omega;C([0,T];H))}
    + \norm[\big]{X_\alpha-Y_\alpha}^2_{L^p(\Omega;L^2(0,T;V))}
    \lesssim \norm[\big]{X_0-Y_0}^2_{L^p(\Omega;H)},
  \]
  which completes the proof noting that
  $\norm{X-Y}_E \leq e^{\alpha T} \norm{X_\alpha-Y_\alpha}_E$.
\end{proof}

Lipschitz continuity of the solution map can also be obtained in the
case $p=0$. As already seen, the space
$E:=C([0,T];H) \cap L^2(0,T;V)$, equipped with the norm
\[
\norm[\big]{u}_E := \norm[\big]{u}_{C([0,T];H)}
+ \norm[\big]{u}_{L^2(0,T;V)},
\]
is a Banach space. Then $L^0(\Omega;E)$, endowed with the topology of
convergence in probability, is a complete metrizable topological
vector space. In particular, the distance
\[
d(f,g) := \E\bigl( \norm{f-g}_E \wedge 1 \bigr)
\]
generates its topology.
\begin{prop}
  If $B$ is Lipschitz-continuous in the sense of assumption
  \emph{(B1)}, then the solution map
  \begin{align*}
    L^0(\Omega;H) &\longto L^0(\Omega;E)\\
    X_0 &\longmapsto X
  \end{align*}
  is Lipschitz-continuous.
\end{prop}
\begin{proof}
  Let $X_0$, $Y_0 \in L^0(\Omega,\cF_0,\P;H)$, and $(X,\xi)$,
  $(Y,\eta)$ the unique solutions in $\cJ_0$ to equation \eqref{eq:0}
  with initial datum $X_0$ and $Y_0$, respectively.  The stopping time
  \[
  T_1 := \inf\Bigl\{ t \geq 0: \, (X-Y)_t^* %
  + \biggl( \int_0^t \norm{X(s)-Y(s)}_V^2\,ds \biggr)^{1/2} \geq 1
  \Bigr\} \wedge T.
  \]
  is well defined thanks to the pathwise continuity of $X$ and $Y$.
  For every $\alpha>0$, using the same notation as in the previous proof,
  Theorem~\ref{thm:Ito} yields, by monotonicity of $\beta$ and
  coercivity of $A$,
  \begin{align*}
    &\bigl( X_\alpha - Y_\alpha \bigr)_{t}^{*2} + \int_0^{t}
    \norm{X_\alpha(s)-Y_\alpha(s)}_V^2\,ds
    + \alpha \int_0^{t} \norm{X_\alpha(s)-Y_\alpha(s)}^2\,ds\\
    &\hspace{3em} \lesssim \norm{X_0-Y_0}^2
    + \int_0^t \norm[\big]{(B(X(s))-B(Y(s)))_\alpha}^2_{\cL^2(U,H)}\,ds\\
    &\hspace{3em} \quad + \bigl( (X_\alpha-Y_\alpha)(B(X)-B(Y))_\alpha \cdot W \bigr)^*_t
  \end{align*}
   Raising to the power $1/2$, stopping at $T_1$, and taking
  expectation, we get, by the Lipschitz continuity of $B$,
  \begin{align*}
    &\E\bigl( X_\alpha - Y_\alpha \bigr)_{T_1}^* %
    + \E\biggl( \int_0^{T_1} \norm{X_\alpha(s)-Y_\alpha(s)}_V^2\,ds \biggr)^{1/2}
    + \sqrt{\alpha} \E\biggl( \int_0^{T_1} \norm{X_\alpha(s)-Y_\alpha(s)}^2\,ds \biggr)^{1/2}\\
    &\hspace{3em} \lesssim \E\ind{[\![0,T_1]\!]}\norm{X_0-Y_0}
    + \E\biggl( \int_0^{T_1} \norm{X_\alpha(s)-Y_\alpha(s)}^2\,ds \biggr)^{1/2}\\
    &\hspace{3em} \quad + \E\bigl( (X_\alpha-Y_\alpha)(B(X)-B(Y))_\alpha \cdot W
    \bigr)^{*1/2}_{T_1},
  \end{align*}
  where, by Lemma~\ref{lm:BDGY} and Lipschitz continuity of $B$, the
  last term on the right-hand side is bounded by
  \[
  \varepsilon \E\bigl( X_\alpha - Y_\alpha \bigr)_{T_1}^* %
  + N(\varepsilon) \E\biggl( \int_0^{T_1} \norm{X_\alpha(s)-Y_\alpha(s)}^2\,ds \biggr)^{1/2}
  \]
  for every $\varepsilon>0$. Therefore, choosing $\varepsilon$ small
  enough and $\alpha$ large enough, we are left with
  \[
  \E\bigl( X - Y \bigr)_{T_1}^* + \E\biggl( \int_0^{T_1}
  \norm{X(s)-Y(s)}_V^2\,ds \biggr)^{1/2} \lesssim 
  \E\ind{[\![0,T_1]\!]}\norm{X_0-Y_0}.
  \]
  The proof is concluded noting that, by definition of $T_1$,
  \[
  \bigl( X - Y \bigr)_{T_1}^* + \biggl( \int_0^{T_1}
  \norm{X(s)-Y(s)}_V^2\,ds \biggr)^{1/2} =
  \norm[\big]{X-Y}_{C([0,T];H) \cap L^2(0,T;V)} \wedge 1,
  \]
  and $\norm{X_0-Y_0} \leq 1$ on $[\![0,T_1]\!]$, hence
  \[
  \E\ind{[\![0,T_1]\!]}\norm{X_0-Y_0}
  = \E\ind{[\![0,T_1]\!]} \bigl( \norm{X_0-Y_0} \wedge 1 \bigr)
  \leq \E\bigl( \norm{X_0-Y_0} \wedge 1 \bigr).
  \qedhere
  \]
\end{proof}


\section{A regularity result}
\label{sec:reg}
We are going to show that the regularity of the solution to equation
\eqref{eq:0} improves, if the initial datum and the diffusion
coefficient are smoother, irrespective of the (possible) singularity
of the drift coefficient $\beta$. In particular, we provide sufficient
conditions implying that the variational solution to \eqref{eq:0} is
also an analytically strong solution, in the sense that it takes
values in the domain of the part of $A$ in $H$ (see
{\S}\ref{sec:cont}). If the solution to \eqref{eq:0} generates a
Markovian semigroup on $C_b(H)$ admitting an invariant measure, we
also show that improved regularity of the solution carries over to
further regularity of the invariant measure, in the sense that its
support is made of smoother functions.
\begin{thm}
  \label{thm:reg1}
  Assume that the hypotheses of {\S}\ref{ssec:ass} are satisfied, that $A$ is symmetric and
  that
  \begin{equation}
    \label{hyp_reg}
    X_0\in L^2(\Omega,\cF_0,\P; V), \qquad
    B(\cdot,X) \in L^2(\Omega; L^2(0,T; \cL^2(U,V))).
  \end{equation}
  Then the unique solution $(X,\xi)$ to the equation \eqref{eq:0} satisfies
  \[
    X \in L^2(\Omega;C([0,T];H))\cap L^2(\Omega; L^\infty(0,T; V))
    \cap L^2(\Omega; L^2(0,T; \dom(A_2))).
  \]
\end{thm}
For the proof we need the following positivity result.
\begin{lemma}\label{lm:pos}
  Let $A_\lambda$ and $\beta_\lambda$ be the Yosida approximations of
  $A_2$ and $\beta$, respectively. One has
  \[
    \ip[\big]{A_\lambda u}{\beta_\lambda(u)} \geq 0 \qquad \forall u \in H.
  \]
\end{lemma}
\begin{proof}
  Let $j_\lambda:\erre \to \erre$ be the positive convex function
  defined as $j_\lambda(x):=\int_0^x \beta_\lambda(y)\,dy$. Then, for
  any $u$, $v \in L^2(D)$,
  \[
  j_\lambda(v)-j_\lambda(u) \geq j_\lambda'(u)(v-u)
  \]
  a.e. in $D$, hence, integrating over $D$,
  \[
    \int_D j_\lambda(v) - \int j_\lambda(u)
    \geq \ip[\big]{\beta_\lambda(u)}{v-u}.
  \]
  Choosing $v=(I+\lambda A_2)^{-1}u$, one has
  $u-v = \lambda A_\lambda u$, thus also
  \[
    \lambda \ip[\big]{\beta_\lambda(u)}{A_\lambda u}
    \geq \int j_\lambda(u) - \int j_\lambda((I+\lambda A_2)^{-1}u).
  \]
  Since $A_1$ is an extension of $A_2$ and $u \in L^1(D)$, Jensen's
  inequality for sub-Markovian operators and accretivity of $A_1$ in
  $L^1(D)$ imply
  \[
    \int j_\lambda((I+\lambda A_2)^{-1}u)
    \leq \int (I+\lambda A_2)^{-1} j_\lambda(u)
    \leq \int j_\lambda(u).
    \qedhere
  \]
\end{proof}

\begin{proof}[Proof of Theorem~\ref{thm:reg1}]
  For any $\lambda>0$, let $J_\lambda$ and $A_\lambda$ be the
  resolvent and the Yosida approximations of $A_2$, the part of $A$ in
  $H$, as defined in {\S}\ref{sec:cont}. That is,
  \[
  J_\lambda := (I+\lambda A_2)^{-1}, \qquad
  A_\lambda := \frac{1}{\lambda}(I-J_\lambda).
  \]
  We recall that $J_\lambda$ and $A_\lambda$ are bounded linear
  operators on $H$, that $J_\lambda$ is a contraction, and that
  $A_\lambda=AJ_\lambda$.

  Setting $G:=B(\cdot,X)$, let us consider the equation
  \[
    dX_\lambda(t) + A_\lambda X_\lambda(t)\,dt + \beta_\lambda(X_\lambda(t))\,dt = 
    G(t)\, dW(t), \qquad X_\lambda(0)=X_0.
  \]
  Since $A_\lambda$ is bounded and $\beta_\lambda$ is
  Lipschitz-continuous, it admits a unique strong solution
  \[
    X_\lambda \in L^2(\Omega;C([0,T];H)),
  \]
  for which It\^o's formula for the square of the $H$-norm yields
  \begin{align*}
    &\frac12 \norm{X_\lambda(t)}^2
      + \int_0^t\ip[\big]{A_\lambda X_\lambda(s)}{X_\lambda(s)}\,ds
      +\int_0^t\!\!\int_D\beta_\lambda(X_\lambda(s))X_\lambda(s)\,ds\\
    &\hspace{3em} = \frac12\norm{X_0}^2
      + \frac12 \int_0^t \norm{G(s)}^2_{\cL^2(U,H)}\,ds
      + \int_0^t X_\lambda(s)G(s)\,dW(s)
  \end{align*}
  for all $t \in [0,T]$ $\P$-almost surely. Writing
  \[
    X_\lambda = J_\lambda X_\lambda + X_\lambda - J_\lambda X_\lambda
    = J_\lambda X_\lambda + \lambda A_\lambda X_\lambda
  \]
  and recalling that $A_\lambda = AJ_\lambda$ and that $A$ is
  coercive, we have, after taking supremum in time and expectation,
  \begin{align*}
    &\frac12 \E \norm[\big]{X_\lambda}^2_{C([0,T];H)}
      + C \E\int_0^T \norm[\big]{J_\lambda X_\lambda(s)}_V^2\,ds\\
    &\hspace{3em} \qquad + \lambda \E\int_0^T \norm{A_\lambda X_\lambda(s)}^2\,ds
      + \E\int_0^T\!\!\int_D \beta_\lambda(X_\lambda(s))X_\lambda(s)\,ds\\
    &\hspace{3em} \leq \frac12 \E\norm{X_0}^2
      + \frac12\E\int_0^T \norm[\big]{G(s)}^2_{\cL^2(U,H)}\,ds
      +\E\sup_{t\in[0,T]} \abs[\bigg]{\int_0^t X_\lambda(s)G(s)\,dW(s)},
  \end{align*}
  where, by Lemma~\ref{lm:BDGY}, the last term on the right-hand side
  is bounded by
  \[
    \varepsilon \E \norm{X_\lambda}^2_{C([0,T];H)}
    + C_\varepsilon \E\int_0^T \norm[\big]{G(s)}^2_{\cL^2(U,H)}\,ds,
  \]
  so that, rearranging terms and choosing $\varepsilon$ small enough,
  we deduce that there exists a constant $N>0$ independent of
  $\lambda$ such that
  \begin{equation}
    \label{estimates1}
    \norm[\big]{X_\lambda}_{L^2(\Omega;C([0,T];H))}^2
    + \norm[\big]{J_\lambda X_\lambda}^2_{L^2(\Omega;L^2(0,T;V))}
    +\norm[\big]{%
      \beta_\lambda(X_\lambda)X_\lambda}_{L^1(\Omega\times(0,T)\times D)}
    < N.
  \end{equation}
  Moreover, let us introduce the function
  \begin{align*}
    \varphi_\lambda : H &\longto [0,+\infty[,\\
    u &\longmapsto \frac12 \ip{A_\lambda u}{u}.
  \end{align*}
  The linearity, boundedness and symmetry of $A_\lambda$ immediately imply
  that $\varphi_\lambda \in C^2(H)$ with
  $D\varphi_\lambda(u) = A_\lambda$, and, by linearity,
  $D^2\varphi_\lambda(u)=A_\lambda$, for all $u \in H$ (in the latter
  statement $A_\lambda$ has to be considered as an element of
  $\cL_2(H)$, the space of bounded bilinear maps on $H$).  It\^o's
  formula applied to $\varphi_\lambda(X_\lambda)$ then yields
  \begin{align*}
    &\varphi_\lambda(X_\lambda(t))
      + \int_0^t\norm[\big]{A_\lambda X_\lambda(s)}^2\,ds + 
    \int_0^t \ip[\big]{A_\lambda X_\lambda(s)}{\beta_\lambda(X_\lambda(s))}\,ds\\
    &\hspace{3em} = \varphi_\lambda(X_0)
      + \frac12 \int_0^t \operatorname{Tr} \bigl(%
      G^*(s) D^2\varphi_\lambda(X_\lambda(s)) G(s) \bigr)\,ds
    + \int_0^t A_\lambda X_\lambda(s) G(s)\,dW(s)
  \end{align*}
  for every $t\in[0,T]$ $\P$-almost surely. Writing, as before,
  $X_\lambda = J_\lambda X_\lambda + \lambda A_\lambda X_\lambda$, the
  coercivity of $A$ implies that
  \[
    \varphi_\lambda(X_\lambda) = \frac12 \ip[\big]{A_\lambda X_\lambda}{X_\lambda}
    \geq \frac{C}{2} \norm[\big]{J_\lambda X_\lambda}_V^2
    + \frac12 \lambda \norm[\big]{A_\lambda X_\lambda}^2
    \gtrsim \norm[\big]{J_\lambda X_\lambda}_V^2.
  \]
  The continuity of $J_\lambda$ on $V$ (see Lemma~\ref{lm:cvV}) instead implies
  \[
    \varphi_\lambda(X_0) = \ip[\big]{AJ_\lambda X_0}{X_0}
    \leq \norm[\big]{A}_{\cL(V,V')} \norm[\big]{J_\lambda X_0}_V
    \norm[\big]{X_0}_V
    \lesssim \norm[\big]{A}_{\cL(V,V')} \norm[\big]{X_0}_V^2.
  \]
  Denoting a complete orthonormal basis of $U$ by $(u_k)_k$,
  we have, recalling again the continuity of $J_\lambda$ on $V$ and
  that $D^2\varphi_\lambda(u) = A_\lambda$ for all
  $u \in H$,
  \begin{align*}
    \operatorname{Tr} \bigl( G^* D^2\varphi_\lambda(X_\lambda) G \bigr)
    &= \sum_{k=0}^\infty \ip[\big]{%
      G^* D^2\varphi_\lambda(X_\lambda) Gu_k}{u_k}_U
      = \sum_{k=0}^\infty \ip[\big]{%
      A_\lambda G u_k}{Gu_k}\\
  &= \sum_{k=0}^\infty \ip[\big]{AJ_\lambda G u_k}{Gu_k}
    \leq\norm[\big]{A}_{\cL(V,V')} \sum_{k=0}^\infty%
    \norm[\big]{J_\lambda Gu_k}_V \norm[\big]{Gu_k}_V\\
  &\lesssim \norm[\big]{A}_{\cL(V,V')} \sum_{k=0}^\infty \norm[\big]{Gu_k}_V^2
  =\norm[\big]{A}_{\cL(V,V')} \norm[\big]{G}_{\cL^2(U,V)}^2.
  \end{align*}
  Moreover, by Lemma~\ref{lm:BDGY},
  \begin{align*}
    \E \bigl( (A_\lambda X_\lambda G) \cdot W \bigr)^*_T
    &\lesssim \varepsilon \E \sup_{t\in[0,T]}
      \norm[\big]{A_\lambda X_\lambda(t)}^2_{V'}
      + N(\varepsilon) \E\int_0^T \norm[\big]{G(s)}^2_{\cL^2(U,V)}\,ds\\
    &\leq \varepsilon \norm[\big]{A}^2_{\cL(V,V')}
      \E \sup_{t\in[0,T]} \norm[\big]{J_\lambda X_\lambda(t)}^2_V
      + N(\varepsilon) \E \int_0^T \norm[\big]{G(s)}^2_{\cL^2(U,V)}\,ds
  \end{align*}
  for every $\varepsilon >0$. 
  Taking supremum in time and expectations in the It\^o formula for
  $\varphi_\lambda(X_\lambda)$, choosing $\varepsilon$ small enough we
  obtain, thanks to the previous lemma and hypothesis \eqref{hyp_reg},
  that there exists a constant $N>0$ independent of $\lambda$, such
  that
  \begin{equation}\label{estimates2}
    \norm[\big]{J_\lambda X_\lambda}^2_{L^2(\Omega; L^\infty(0,T; V))} + 
    \norm[\big]{A_\lambda X_\lambda}^2_{L^2(\Omega; L^2(0,T; H))} < N.
  \end{equation}
  Reasoning as in \cite{cm:AP18}, it follows by 
  \eqref{estimates1} that
  \begin{alignat*}{2}
    X_\lambda &\longto X
    &&\quad \text{ weakly in } L^2(\Omega; L^2(0,T; H)),\\
    J_\lambda X_\lambda &\longto X
    &&\quad \text{ weakly in } L^2(\Omega;L^2(0,T; V)),\\
    \beta_\lambda(X_\lambda) &\longto \xi
    &&\quad \text{ weakly in } L^1(\Omega\times(0,T)\times D),
  \end{alignat*}
  where $(X,\xi)$ is the unique solution to \eqref{eq:0}.  Moreover,
  by \eqref{estimates2} we have
  \[
    J_\lambda X_\lambda - X_\lambda = \lambda A_\lambda X_\lambda
    \longto 0 \qquad\text{ in } L^2(\Omega;L^2(0,T; H)),
  \]
  hence, $\P$-almost surely and for almost every $t \in (0,T)$,
  $J_\lambda X_\lambda(t)$ converges to $X(t)$ in $H$. Since the
  function $\norm{\cdot}_V^2$ is lower semicontinuous on $H$, we infer
  that
  \[
    \norm{X(t)}^2_V\leq \liminf_{\lambda\to 0}
    \norm[\big]{J_\lambda X_\lambda(t)}_V^2
  \]
  for almost every $t$. Hence, taking supremum in time and
  expectation, we deduce that
  \[
    X \in L^2(\Omega;L^\infty(0,T;V)).
  \]
  Moreover, by \eqref{estimates2} we
  also have
  \[
    A_\lambda X_\lambda \longto \eta \quad
    \text{ weakly in } L^2(\Omega; L^2(0,T; H)),
  \]
  hence, since $J_\lambda X_\lambda \to X$ weakly in
  $L^2(\Omega; L^2(0,T; V))$, by the continuity and the linearity of
  $A$ we necessarily have $\eta=AX$. In particular,
  $X \in L^2(\Omega;L^2(0,T;\dom(A_2)))$.
\end{proof}

As last result we show that if the solution to \eqref{eq:0} generates
a Markovian semigroup $P=(P_t)_{t\geq 0}$ on $C_b(H)$ admitting an
invariant measure, then the improved regularity of solutions given by
Theorem~\ref{thm:reg1} implies better integrability properties also
for the invariant measures. Existence and uniqueness of invariant
measures, ergodicity, and the Kolmogorov equation associated to
\eqref{eq:0} were studied in \cite{cm:inv}. In particular, we recall
the following result (see \cite[{\S}5]{cm:inv}). The set of invariant
measures of $P$ will be denoted by $\mathscr{M}$.
\begin{prop}
  Assume that the the hypotheses of \S\ref{ssec:ass} are satisfied,
  that $X_0\in H$ is non-random, and that $B$ is non-random and
  time-independent. Then the solution $X$ to \eqref{eq:0} is Markovian
  and its associated transition semigroup $P$ admits an
  ergodic invariant measure.  Moreover, there exists a positive
  constant $N$ such that 
  \[
    \int_H \norm{u}_V^2\,\mu(du)
    + \int_H\!\int_D j(u)\,\mu(du)
    + \int_H\!\int_Dj^* (\beta^0(u))\,\mu(du) < N
    \qquad \forall \mu \in \mathscr{M}.
  \]
  If $\beta$ is superlinear, there exists a unique invariant measure
  for $P$, which is strongly mixing as well.
\end{prop}

\begin{thm}
  Assume that the hypotheses of \S\ref{ssec:ass} are satisfied, that $A$ is symmetric and that 
  \[
    B:H\to\cL^2(U,V), \qquad
    \norm{B(v)}_{\cL^2(U,V)} \lesssim 1 + \norm{v}_V.
  \]
  Then there exists a positive constant $N$ such that 
  \[
  \int_H\norm{Au}^2\,\mu(du) < N \qquad \forall \mu \in \mathscr{M}.
  \]
  In particular, every invariant measure $\mu$ is concentrated on
  $\dom(A_2)$, that is $\mu(\dom(A_2))=1$.
\end{thm}
\begin{proof}
  For every $x\in V$, let $(X^x,\xi^x)$ be the unique strong solution
  to \eqref{eq:0} with initial datum $x$.  Setting $G:=B(X^x)$,
  It\^o's formula for $\varphi_\lambda(X_\lambda)$ as in the proof of
  the previous theorem yields
  \begin{align*}
    &\varphi_\lambda(X_\lambda(t))
      + \int_0^t\norm[\big]{A_\lambda X_\lambda(s)}^2\,ds + 
    \int_0^t \ip[\big]{A_\lambda X_\lambda(s)}{\beta_\lambda(X_\lambda(s))}\,ds\\
    &\hspace{3em} = \varphi_\lambda(x)
      + \frac12 \int_0^t \operatorname{Tr} \bigl(%
      G^*(s) D^2\varphi_\lambda(X_\lambda(s)) G(s) \bigr)\,ds
    + \int_0^t A_\lambda X_\lambda(s) G(s)\,dW(s).
  \end{align*}
  Since $A_\lambda X_\lambda \in L^2(\Omega; L^\infty(0,T; H))$ and
  $G\in L^2(\Omega; L^2(0,T; \cL^2(U,H)))$, the last term on the right
  hand side is a martingale; hence, taking expectations and recalling
  that $\varphi_\lambda(x)\lesssim \norm{x}_V^2$, it follows by
  Lemma~\ref{lm:pos} and by the estimates obtained in the proof of the
  previous theorem that
  \[
  \E\int_0^t\norm[\big]{A_\lambda X_\lambda(s)}^2\,ds
  \lesssim \norm{x}_V^2 + \E\norm{G}^2_{L^2(0,t; \cL^2(U,V))}.
  \]
  Since this holds for every $\lambda>0$, letting $\lambda \to 0$ and
  recalling that, as in the proof of the previous theorem,
  $A_\lambda X_\lambda$ converges to $AX$ weakly in
  $L^2(\Omega;L^2(0,T;H))$, a weak lower semicontinuity argument and
  the linear growth assumption on $B$ yield
  \begin{equation}
    \label{eq:im}
    \E \int_0^t \norm[\big]{AX^x(s)}^2\,ds
    \lesssim 1 + \norm{x}_V^2
  \end{equation}
  for every $t\in[0,T]$ and $x\in V$.
  Let us introduce the function $F:H \to [0,+\infty]$ defined as
  \[
  F(u) := \begin{cases}
    \norm{Au}^2 \quad&\text{if } u\in\dom(A_2),\\
    +\infty \quad&\text{if } u\in H\setminus\dom(A_2),
  \end{cases}
  \]
  and the sequence of functions $(F_n)_{n\in\enne}$, $F_n: H \to [0,+\infty)$,
  defined as
  \[
  F_n(u) := \norm[\big]{A_{1/n}u}^2 \wedge n^2.
  \]
  It is easily seen that $F_n \in C_b(H)$ for all $n \in \enne$ and
  that $F_n$ converges pointwise to $F$ from below. Therefore, for any
  invariant measure $\mu$, it follows by Fubini's theorem that
  \begin{align*}
    \int_H F_n(x)\,\mu(dx)
    &= \int_0^1\!\!\int_H F_n(x)\,\mu(dx)\,ds
    = \int_0^1\!\!\int_HP_sF_n(x)\,\mu(dx)\,ds\\
    &=\int_H\!\int_0^1 \E \bigl( \norm[\big]{A_{1/n}X^x(s)}^2 \wedge n^2 \bigr)%
      \,ds\,\mu(dx)\\
    &\leq \int_H\!\E\int_0^1 \norm[\big]{A_{1/n}X^x(s)}^2 \,ds\,\mu(dx)
  \end{align*}
  Recalling that $\norm{A_\lambda u} \leq \norm{Au}$ for all
  $u \in H$, it follows by \eqref{eq:im} that
  \[
    \int_H F_n(x)\,\mu(dx) \lesssim 1 
    + \int_H \norm{x}_V^2\,\mu(dx).
  \]
  Since $\norm{\cdot}_V^2\in L^1(H,\mu)$ by
  \cite[Theorem~5.3]{cm:inv}, we get
  \[
    \int_H F_n(x)\,\mu(dx) \lesssim  N
  \]
  for a positive constant $N$, independent of $n$ and $\mu$.  Letting
  $n\to\infty$, by the monotone convergence theorem we deduce that
  $F \in L^1(H,\mu)$, hence $F$ is finite $\mu$-almost everywhere in
  $H$, and in particular $\mu(\dom(A_2))=1$.
\end{proof}

\bibliographystyle{amsplain}
\bibliography{ref,extra}

\end{document}